\documentclass[twoside, 10pt]{article}
\usepackage{graphicx} % Required for inserting images
\usepackage{mathalfa}
\usepackage{mathrsfs}
\usepackage{amsmath}
\usepackage{amsthm,enumerate}
\usepackage{indentfirst,graphicx,epsfig}

\input{epsf}

\usepackage{epstopdf}
\usepackage{tikz,pgf}
\usepackage{subfig}
\usepackage{caption}
\usepackage{array}
\usepackage{makecell}
\usepackage{tabularx}
\usepackage{amssymb}

\usetikzlibrary{shapes.geometric} 

\usetikzlibrary{positioning}
\usetikzlibrary{decorations.markings}
\usetikzlibrary{decorations.pathmorphing}
\usetikzlibrary{patterns}
\tikzset{->-/.style={decoration={
  markings,
  mark=at position .5 with {\arrow[scale=0.8]{>}}},postaction={decorate}}}
\tikzset{snake it/.style={decorate, decoration={snake, amplitude=.4mm, segment length=2mm}}}
\usepackage{pgfplots}
\pgfplotsset{compat=1.15}
\usepackage{mathrsfs}
\usetikzlibrary{arrows}
\definecolor{wqwqwq}{rgb}{0.3764705882352941,0.3764705882352941,0.3764705882352941}
\definecolor{uuuuuu}{rgb}{0.26666666666666666,0.26666666666666666,0.26666666666666666}
\definecolor{uququq}{rgb}{0.25098039215686274,0.25098039215686274,0.25098039215686274}

\newtheorem{thm}{Theorem}[section]

\newtheorem{lem}[thm]{Lemma}
\newtheorem{pro}[thm]{Proposition}
\newtheorem{obs}[thm]{Observation}

\newtheorem{claim}[thm]{Claim}

\title{\textbf{Planar Tur\'{a}n numbers of three configurations}}
\author{Xuqing Bai$^1$ \quad\quad Zhipeng Gao$^1$\quad\quad Ping Li$^{2,}$\footnote{Corresponding anthor.}\\
{\small 1. School of Mathematics and Statistics, Xidian University, Xi'an, 710071, China}\\
{\small baixuqing@xidian.edu.cn; gaozhipeng@xidian.edu.cn}\\
{\small 2. School of Mathematics and
Statistics, Shaanxi Normal University, Xi'an, Shaanxi, China.} \\
{\small lp-math@snnu.edu.cn}\\ 
}
\date{}

\begin{document}

\maketitle
\begin{abstract}\baselineskip 0.50cm
    The planar Tu\'{a}n number of $H$, denoted by $ex_{\mathcal{P}}(n,H)$, is defined as the maximum number of edges in an $n$-vertex $H$-free planar graph. The exact value of $ex_{\mathcal{P}}(n,H)$ remains a mystery when $H$ is large (for example, $H$ is a long path or a long cycle), while tight bounds have been established for many small planar graphs such as cycles, paths, $\Theta$-graphs and other small graphs formed by a union of them. 
    One representative graph among such union graphs is $K_1+L$ where $L$ is a linear forest without isolated vertices.
    Previous works solved the cases in which  $L$ is a path or a matching, or satisfies  $|L|\geq 7$.
    In this work, we first investigate the planar Tur\'{a}n number of the graph $K_1+L$ when $L$ is the disjoint union of a $P_2$ and $P_3$.
    Equivalently,  $K_1+L$  represents a specific configuration formed by combining a $C_3$ and a $\Theta_4$.
    We further consider the planar Tur\'{a}n numbers of the all graphs obtained by combining $C_3$ and $\Theta_4$. Among the six possible such configurations, three have been resolved in earlier works. For the remaining three configurations (including 
     $K_1+(P_2\dot{\cup}P_3)$), we derive tight bounds.    Furthermore, we completely characterize all extremal graphs for the remaining two of these three cases.
\end{abstract}

\baselineskip 0.50cm

\section{Introduction}
A graph $G$ is called {\em $G$-free} if it does not contain $G$ as a subgraph. The Tur\'{a}n number of a graph $G$, denoted by $ex(n,G)$, is the maximum number of edges in an $n$-vertex $G$-free graph. Tur\'{a}n-type problems are central topics in extremal combinatorics, employing diverse methodologies and intersecting with multiple mathematical disciplines. In 2016, Dowden \cite{dowden2016extremal} initiated the study of pl anar Tur\'{a}n-type problems. The planar Tur\'{a}n number of $H$, denoted by $ex_{\mathcal{P}}(n, H)$, is the maximum number of edges in an $n$-vertex $H$-free planar graph. 
Since planar graphs constitute a special and highly sparse graph class, the study of planar Tur\'{a}n problems relies primarily on structural methods.

Many planar graphs, particularly large and dense ones, have a planar Tur\'{a}n number of $3n-6$. This trivial value is achieved because triangulations can be constructed to avoid specific forbidden subgraphs. For example, if $H$ contains at least three vertex-disjoint cycles, then $ex_{\mathcal{P}}(n, H)=3n-6$, as demonstrated by the triangulation $K_2+ P_{n-2}$. Lan, Shi, and Song \cite{lan2019extremalTheta} provided sufficient conditions for a graph to have this trivial planar Tur\'{a}n number, a complete characterization remains unknown. 
One such condition is the maximum degree exceeds six (in other words, $ex_{\mathcal{P}}(n, H)=3n-6$ is $\Delta(H)\geq 7$). This makes the study of planar Tur\'{a}n numbers particularly interesting for two classes of graphs: those that are highly sparse and those that are small in terms of order.

It is well-known that the existence of Hamiltonian cycles in planar graph is a mystery, which results in a particularly interesting but challenging problem for determining the exact value of planar Tur\'{a}n numbers of long paths and long cycles.
Shi, Walsh, and Yu \cite{shi2025dense} proved that $ex_{\mathcal{P}}(n,C_k)\leq 3n-6-4^{-1}k^{\log_32}n$ for large $k$. Combined with a lower bound by Gy\H{o}ri, Varga and Zhu  \cite{gyHori2024new}, the planar Tur\'{a}n number for long cycles $C_k$ is $(3-\Theta(k^{\log_23}))n$. Li \cite{li2025dense} also proved that $ex_{\mathcal{P}}(n,2C_k)=(3-\Theta(k^{\log_23}))n$.
For a large planar graph $H$, if its planar Tur\'{a}n is not the trivial value $3n-6$, then it must contain a long path or cycle since $\Delta(H)\leq 6$.
Therefore,  determining the planar Tur\'{a}n numbers of long paths and long cycles is an elementary problem.

Given the aforementioned challenges, researchers have shifted their focus toward determining the exact planar Tur\'{a}n numbers of small planar graphs. So far, tight bounds are known only for 
  a few small planar graphs, including short cycles of length up to seven \cite{dowden2016extremal,ghosh2022planar,gyHori2023planar}, paths of order at most eleven  \cite{lan2019planar}, and small $\Theta$-graphs \cite{lan2019extremalTheta} (we use $\Theta_k$ to denote a set of graphs that are obtained from a cycle by adding an additional edge joining two non-consecutive vertices. It is clear that $\Theta_k$ is a single graph if $k=4,5$). Other studies have examined structures composed of combined cycles, such as unions of triangles sharing vertices or edges  \cite{fang2022planar,lan2024planar} and disjoint union of graphs \cite{fang2024extremal,lan2024planar}.
Among these configurations, one of the most interesting is $K_1+L_t$, where $L_t$ is a linear forest of $t$ vertices (the graph $K_1+L_t$ is obtained by joining a new vertex to all the
vertices of $L_t$).
Lan, Shi, and Song \cite{lan2018extremal} proved that ffor an integer $4 \le t \le 6$, let $H$ be a graph on $t$ vertices consisting of disjoint paths. Then $ex_{\mathcal{P}}(n, K_1 + H) \le \frac{13(t-1)n}{4t-2} - \frac{12(t-1)}{2t-1}$ for all $n \ge t+1$.
It is worth noting that not all the upper bounds obtained are tight.  They further determined a tight upper bound for $ex_\mathcal{P}(n, K_1+2P_2)$ when $n\geq 5$, and provided an improved upper bound for $ex_\mathcal{P}(n, K_1+3P_2)$. 
Subsequently, Fang, Wang, and Zhai \cite{fang2022planar} established the tight bound for $ex_\mathcal{P}(n, K_1+3P_2)$ and the tight bound of $ex_\mathcal{P}(n, K_1+P_t)$ for $3\leq k\leq 6$.
In this paper, we extend this line of research by studying graphs of the form $K_1+H$. Specifically, we establish a tight bound for  $ex_\mathcal{P}(n, K_1+H)$ when $H=P_2\dot{\cup}P_3$ is a disjoint union  of $P_2$ and $P_3$.

\begin{thm}\label{thm1}
	$ex_{\mathcal{P}}(n, K_1+(P_2\dot{\cup} P_3))\leq \frac{13n}{5}-\frac{26}{5}$ for all $n\geq 72$, with equality if $n=2\pmod 5$.
\end{thm}

Note that $K_1+(P_2\dot{\cup}P_3)$ corresponds to a specific configuration formed by combining $C_3$ and $\Theta_4$.
We further investigate all possible unions of $C_3$ and $\Theta_4$. It is clear that there are six different combinations illustrated in Figure 1 ($H_4=K_1+(P_2\dot{\cup}P_3)$).
\begin{figure}[ht]\label{Fig1}
     \centering  
 \hspace{1em} \subfloat[{$H_1$}]
    {\begin{tikzpicture}
    [inner sep=0.1mm]	
               \node[circle, fill, inner sep=1.5pt](v0) at (0,0)[]{};
               \node[circle, fill, inner sep=1.5pt](v1) at (1.73,0)[]{};
                \node[circle, fill, inner sep=1.5pt](v2) at (0.865,3/2)[]{};
                \node[circle, fill, inner sep=1.5pt](v3) at (0.865,1/2)[]{};
                 
                \draw[-] (v0) -- (v1);
                \draw[-] (v1) -- (v2);
                \draw[-] (v2) -- (v0);
                \draw[-] (v3) -- (v1);
                \draw[-] (v3) -- (v2);
                \draw[-] (v3) -- (v0);
            \end{tikzpicture}
    }  \hspace{1em} 
    \subfloat[{$H_2$}]
    {\begin{tikzpicture}
    \pgfmathparse{sqrt(3)}
    [inner sep=0.8pt]	
                \node[circle, fill, inner sep=1.5pt](v0) at (0,0)[]{};
               \node[circle, fill, inner sep=1.5pt](v1) at (1.73,0)[]{};
                \node[circle, fill, inner sep=1.5pt](v2) at (0.86,3/2)[]{};
                \node[circle, fill, inner sep=1.5pt](v3) at (0.57,0)[]{}; 
                \node[circle, fill, inner sep=1.5pt](v4) at (1.154,0)[]{};
                \draw[-] (v0) -- (v3) -- (v4) -- (v1);
                \draw[-] (v0) -- (v2) -- (v3);
                \draw[-] (v4) -- (v2) -- (v1);
            \end{tikzpicture}
    }
    \hspace{1em} \subfloat[{$H_3$}]
   {\begin{tikzpicture}
    \pgfmathparse{sqrt(3)}
    [inner sep=0.1mm]	
               \node[circle, fill, inner sep=1.5pt](v0) at (0,0.38)[]{};
               \node[circle, fill, inner sep=1.5pt](v1) at (1.73,0.38)[]{};
                \node[circle, fill, inner sep=1.5pt](v2) at (0.865,3/2)[]{};
                \node[circle, fill, inner sep=1.5pt](v3) at (0.865,0.9)[]{};
                \node[circle, fill, inner sep=1.5pt](v4) at (0.865,0)[]{};
                 
                \draw[-] (v0) -- (v1) -- (v2) -- (v0);
                \draw[-] (v2) -- (v3);
                \draw[-] (v0) -- (v3) -- (v1);
                \draw[-] (v0) -- (v4) -- (v1);
            \end{tikzpicture}
            }
    \hspace{1em} \subfloat[{$H_4$}]
    {\begin{tikzpicture}
    [inner sep=0.8pt]	
               \node[circle, fill, inner sep=1.5pt](u1) at (-0.85,0.5)[]{};
                \node[circle, fill, inner sep=1.5pt](u2) at (-0.85,-0.5)[]{};
               \node[circle, fill, inner sep=1.5pt](v1) at (0,0)[]{};
                \node[circle, fill, inner sep=1.5pt](v2) at (0.433,-0.75)[]{};
                \node[circle, fill, inner sep=1.5pt](v3) at (0.866,0)[]{};
                 \node[circle, fill, inner sep=1.5pt](v4) at (0.433,0.75)[]{};
               
                \draw[-] (v1) -- (v2) -- (v3) -- (v4) -- (v1) -- (v3);
                \draw[-] (u1) -- (u2) -- (v1) -- (u1);
       % \node[below] at (-0.425,-1.3) {$H_1$};
            \end{tikzpicture}
    }
    \hspace{1em} \subfloat[{$H_5$}]
   {\begin{tikzpicture}
    [inner sep=0.1mm]	
                      \node[circle, fill, inner sep=1.5pt](u1) at (-0.85,0.5)[]{};
                \node[circle, fill, inner sep=1.5pt](u2) at (-0.85,-0.5)[]{};
               \node[circle, fill, inner sep=1.5pt](v1) at (0,0)[]{};
                \node[circle, fill, inner sep=1.5pt](v2) at (0.577,-0.75)[]{};
                \node[circle, fill, inner sep=1.5pt](v3) at (1.1547,0)[]{};
                 \node[circle, fill, inner sep=1.5pt](v4) at (0.577,0.75)[]{};
                \draw[-] (v2) -- (v3) -- (v4) -- (v1) -- (v2) -- (v4);
                \draw[-] (u1) -- (u2) -- (v1) -- (u1);
                % \node[below] at (-0.425,-1.3) {$H_2$};
            \end{tikzpicture}
    }  \hspace{1em}  \subfloat[{$H_6$}]
    {\begin{tikzpicture}
    [inner sep=0.8pt]	
               \node[circle, fill, inner sep=1.5pt](u1) at (-0.633,0.45)[]{};
                \node[circle, fill, inner sep=1.5pt](u2) at (-1.066,-0.3)[]{};
                \node[circle, fill, inner sep=1.5pt](u3) at (-0.2,-0.3)[]{};
               \node[circle, fill, inner sep=1.5pt](v1) at (0,0)[]{};
                \node[circle, fill, inner sep=1.5pt](v2) at (0.433,-0.75)[]{};
                \node[circle, fill, inner sep=1.5pt](v3) at (0.866,0)[]{};
                 \node[circle, fill, inner sep=1.5pt](v4) at (0.433,0.75)[]{};
               
                \draw[-] (v1) -- (v2) -- (v3) -- (v4) -- (v1) -- (v3);
                \draw[-] (u1) -- (u2) -- (u3) -- (u1);
       % \node[below] at (-0.425,-1.3) {$H_1$};
            \end{tikzpicture}
    }
    \caption{\label{fig:1} Six types of combinations of $C_3$ and $\Theta_4$.}
\end{figure}
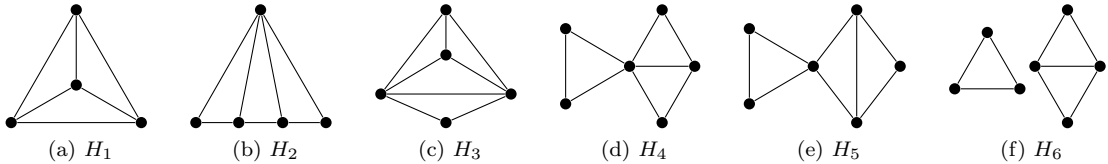
Dowden \cite{dowden2016extremal} determined that $ex_{\mathcal{P}}(n,H_1)=3n-6$ when $n\geq 6$. Fang, Wang, and Zhai \cite{fang2022planar} showed that $ex_{\mathcal{P}}(n,H_2)\leq \frac{8}{3}(n-2)$ with a sharp upper bound. Because $H_1$ is a subgraph of $H_3$, $ex_{\mathcal{P}}(n,H_3)=3n-6$ follows directly. In addition to  $H_4$, we also focus on determining the planar Tur\'{a}n numbers of the  remaining two graphs $H_5$ and $H_6$.

For $H_5$, we determine its planar Tur\'{a}n number exactly when $n=10x+6y$ has integer solutions  $x\geq 2$ and $y\geq 0$, and characterize all extremal graphs.

\begin{thm}\label{thm2}
$ex_{\mathcal{P}}(n, H_5)\leq \lfloor \frac{5n}{2}\rfloor-4$ for all $n\geq 6$, with equality if $n=10x+6y$ has integer solutions  $x\geq 2$ and $y\geq 0$.
Moreover, the extremal graphs can be characterized (see Remark 1).
\end{thm}

We determine the planar Turán number of $H_6$, which is the disjoint union of $C_3$ and $\Theta_4$, and characterize all extremal graphs.
For this purpose, we first introduce two graphs. We use $M_t$ to denote the union of $\left\lfloor\frac{t}{2}\right\rfloor$ pairwise vertex-disjoint edges and $\left\lceil\frac{t}{2}\right\rceil- \left\lfloor\frac{t}{2}\right\rfloor$ isolated vertices. 
For odd $n$, let $K_2\vee M_{n-2}$ denote the graph obtained from $K_2+M_{n-3}$ by adding an additional vertex $u$ and two edges $uv_1,uv_2$, where $v_1$ and $v_2$ are endpoints of two arbitrary edges in $M_{n-2}$, respectively.
To describe the extremal $H_6$-free planar graphs,  we also need the notion outerplanar Tur\'{a}n number of $C_3$ (denoted $ex_{\mathcal{OP}}(n,C_3)$), which is the maximum number of edges in an $n$-vertex $C_3$-free outerplanar graph.
Fang and Zhai \cite{FZ-outerplanar} proved that $ex_{\mathcal{OP}}(n,C_3)=\frac{3n-4}{2}$ for each even integer $n\geq 4$.

\begin{thm}\label{thm: C3Theta4}
If $n\geq 174$, then $ex_{\mathcal{P}}(n, C_3\dot{\cup} \Theta_4)=\left\lfloor \frac{5n}{2}\right\rfloor-4$.
Moreover, if $n$ is even, then the extremal planar graph is a copy of $M_{n-2}+K_2$; if $n$ is odd, then the extremal planar  graph is either a copy of $K_2+M_{n-2}$, or $K_2\vee M_{n-2}$,  or $\{u\}+O$, where $O$ represents a $C_3$-free outerplanar graph of even order with $ex_{\mathcal{OP}}(n-1,C_3)=\lfloor \frac{3n}{2}\rfloor-3$ edges.
\end{thm}

The rest of this paper is organized as follows. In Section \ref{sec:2}, we introduce some concepts and notation on planar graphs.
In Sections \ref{sec:3}, \ref{sec:4} and \ref{sec:5}, we establish sharp bounds on the planar Tur\'{a}n number of $H_4,H_5$ and $H_6$, respectively, and characterize all extremal planar graphs for the latter two graphs.  

% Lan et al. \cite{lan2018extremal} described several sufficient conditions on $K_4$-free planar graphs $H$ such that $ex_{\mathcal{P}}(n,H)=3n-6$ for all $n\geq |H|$, and it seems non-trivial to determine $ex_\mathcal{P}(n,H)$ for all $K_4$-free planar graphs
% $H$ with exactly one vertex, say $u$, satisfying $d_H (u) = \Delta(H)\leq 6$ and $\Delta(H[N(u)])\leq 2$. Moreover, they gave an upper bound (but not tight) for $ex_\mathcal{P}(n, K_1+H)$, where $H$ is a disjoint union of paths.
% \begin{thm}[\cite{lan2018extremal}]\label{sec1-lan-path}
% Let $4\leq t\leq 6$ be an integer, and let $H$ be a graph on $t$ vertices such that $H$ is a disjoint union of paths. 
% Then $ex_{\mathcal{P}}(n,K_1 + H)\leq \frac{13(t-1)n}{4t-2}-\frac{12(t-1)}{2t-1}$ for all $n\geq t+1$.  
% \end{thm} 
% When $H=P_2\cup P_3$, we get $K_1 + H$ is $H_4$ in Figure \ref{fig:1}. 
% In this paper, we complete the sharpness of the upper bounds of $ex_{\mathcal{P}}(n, H_4)$ and $ex_{\mathcal{P}}(n, H_5)$, and determined $ex_{\mathcal{P}}(n, H_6)$ for large enough $n$. 

%%%%%%%%%%%%%%%%%%%%%%%%%%%%%%%%%%%%%%%%%%%%%%%%%%%%%%%%%%%%%%%%%%%%%%%%%%%%%%%%%%%%%%%%%
\section{Preliminaries}\label{sec:2}
%  We present essential definitions and preliminary results required for the proofs in this paper.
% A {\em plane graph} is a planar embedding of a planar graph. Let $G$ be a plane graph with the vertex set $V(G)$ and edge set $E(G)$.
% % The boundary of a face $f$ is the boundary of the open set $f$ in the usual topological sense. 
% A face of $G$ is said to be {\em incident} with the vertices and edges in its boundary, and two faces are {\em adjacent} if their boundaries have an edge in common.
% The degree, $d(f)$, of a face $f$ is the number of edges in its boundary $\partial(f)$.
% Each plane graph has exactly one unbounded
% face, called the {\em outer face}.
% Any face in a plane graph that is not the outer face is referred to as an {\em inner face}.
% A {\em near-triangulation} is a plane graph all of whose inner faces have degree three.
% If an edge $e$ of $G$ lies in the boundaries of exactly two 3-faces, then we call $e$ an interior
% edge of $G$. We use $E_I(G)$ to denote the set of all interior edges of $G$ and $E_I(v)$ to denote the set of all interior edges incident with $v$. For an edge $e=uv$ of
% $E_I(G)$, we call $F_1\cup F_2$ a $\Theta$-graph of $uv$ and denote it by $\Theta_{uv}$ or $\Theta_e$, where $F_1$ and $F_2$ are the two 3-faces of $G$ whose boundaries contain $e$.
% Let $W_{k}=K_1+C_{k}$ be a {\em $k$-wheel}, which is obtained by joining a new vertex to all the vertices of $C_{k}$. 

We first present essential definitions and preliminary results. A graph is \textit{planar} if it can be drawn in the plane without edges crossing except at vertices. Such a drawing is a planar embedding, and a planar graph with a planar embedding is called a \textit{plane graph}. A \textit{chord} of a cycle $C$ in a graph $G$ is an edge not in $C$ whose endpoints both lie on $C$.
For a face $F$ of a connected plane graph $G$, we use $\partial(F)$ to denote the boundary of $F$, which is a closed walk.
If the length of this closed walk if $k$, then we call $F$ a $k$-face.
Specifically, if $G$ is $2$-connected, then $\partial(F)$ is a cycle, and we call it the {\em facial cycle} of $F$.
The {\em outer boundary} of $G$ is the boundary of its outer face.
We always use $f_k(G)$ to denote the number of $k$-faces in the connected plane graph $G$.

A plane graph has a single unbounded \textit{outer face}, with all other faces being \textit{inner faces}. A \textit{near-triangulation} is a $2$-connected plane graph where every inner face has degree three. A vertex with degree $k$ in $G$ is called a \textit{$k$-vertex}. 
Let \( E_I(G) \) be the set of all edges of \( G \) that incident with two $3$-faces, and let \( E_I(v) \) be the set of edges in \( E_I(G) \) containing \( v \) as an end vertex. For an edge \( e = uv\in E_I(G) \), the union of its two adjacent 3-faces, say \( F_1 \) and \( F_2 \), forms a \(\Theta\)-graph of \( uv \), denoted \( \Theta_{uv} \) or \( \Theta_e \). 
A {\em \(k\)-wheel}, denoted by \(W_k\), is formed by connecting a single vertex to all vertices of a \(k\)-cycle. A {\em \(k\)-fan}  is formed by connecting a single vertex to all vertices of a \(k\)-path. We refer to $K_1+t K_2$ as a \textit{friendship graph}, where $t$ is an integer.

For a plane graph $G$ and two distinct inner 3-faces $F_0$ and $F_\ell$, we say  $F_0$ is {\em triangular-connected} to $F_\ell$ (denoted $F_0 \sim F_\ell$) if there exists an alternating sequence $F_0e_1F_1e_2\ldots e_\ell F_\ell$ such that $e_i$ is incident with both $F_i$ and $F_{i-1}$. For a 3-face $F$, let $\hat{F}$ be the set of inner 3-faces of $G$ that are triangular-connected to $F$. 
The {\em triangular-block} (TB for brief) $[\hat{F}]$ is the plane subgraph induced by the vertices and edges of the 3-faces in $\hat{F}$.
Observe that TBs are edge-disjoint. Two TBs are adjacent if they share vertices in $G$, and we call such vertices \textit{junction vertices}. A {\em triangular-component} (or shortly TC) is recursively constructed as follows:
\begin{enumerate}
    \item Initialize with a TB $H=H_0$ of $G$.
    \item Iteratively append any TB $H_i$ adjacent to $H$, updating $H=H\cup H_i$. 
    \item Terminate when no further adjacent TBs exist.
\end{enumerate}

An inner face $F$ of the TB $B$ (resp. the TC $C$) that is not a face of $G$ is called a \textit{hole} of $B$ (resp. a hole of $C$). Note that a hole in $B$ or $C$ may be a 3-face of $B$ or $C$ but not a 3-face of $G$. Let $S$ be a subgraph of $G$. We denote by $\Delta_S$ the number of inner 3-faces of $S$ that are also 3-faces of $G$. The \textit{triangle-density} of $S$ is then defined as $\rho(S) = \frac{\Delta_S}{|S|}$. It is important to emphasize that only 3-faces of $S$ that are also 3-faces in $G$ contribute to the computation of $\Delta_S$.

For each $H$-free TB, we can transform it into a new TB by regarding all holes whose facial cycles are $C_3$s with corresponding 3-faces, such that the new TB maintains $H$-free.
We call such a new TB a {\em solid TB} (i.e., a solid TB contains no holes with facial cycle $C_3$).
In the proofs of Theorems \ref{thm1} and \ref{thm2}, we will estimate the maximum triangle-density among all $H$-free TBs and TCs (where $H$ is either $H_4$ or $H_5$).
For this purpose, it suffices to estimate the maximum triangle-density among all $H$-free solid TBs and all $H$-free TCs composed of solid TBs.
See Figure \ref{example} as an example, the left side is a plane graph, whose TBs are $B_1,B_2$ and $B_3$ (see the Figure \ref{example} (2), two gray faces are holes). $B_1'$ is a solid TB obtained from $B_1$ (see the Figure \ref{example} (3)).

\begin{figure}[ht]
	\centering
	\includegraphics[width=350pt]{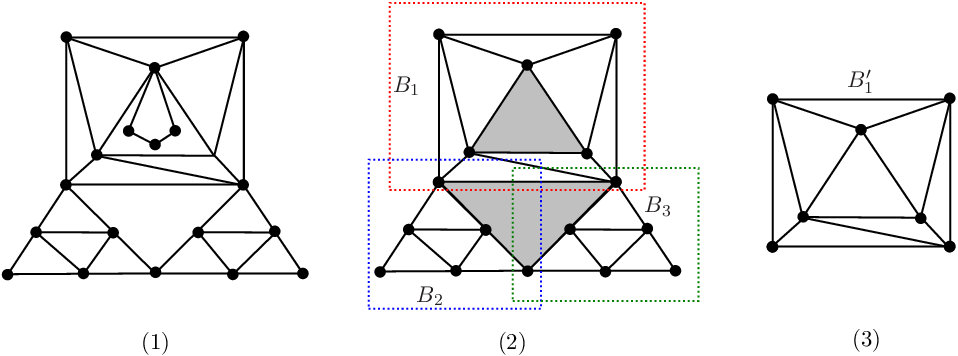}
	 \caption{An example on TBs and solid TBs.} \label{example}
\end{figure}

%\textcolor{red}{We define a \textit{false $k$-face} as a face of length $k$ within the subgraph $[\hat{F}]$ that is not a 3-face in $G$.}Notably, the density may exceed 1 as shown in figure \ref{fig:H4free}, where $\rho(B_9)=\frac{7}{6}$.\textcolor{red}{It is important to emphasize that only 3-faces of $S$ that are also 3-faces in $G$ contribute to the computation of $\Delta_S$.}
%To conclude this section, we present several properties of TBs and TCs.

%%%%%%%%%%%%%%%%%%%%%%%%%%%%%%%%%%%%%%%%%%%%%%%%%%%%%%%%%%%%%%%%%%%%%%%%%%%%%%%%%%%%%%%%%
\section{Proof of Theorem \ref{thm1}}\label{sec:3}

This section focuses on the planar Tur\'{a}n number of $H_4=K_1+(P_2\dot{\cup} P_4)$. 
We begin with a proposition on $H_4$-free solid TBs, followed by a characterization of all such solid TBs.
\begin{pro}\label{sec-1-pro-0} 
%Let $G$ be a plane graph, and 
Let $B$ be an $H_4$-free solid TB with $|B| \geq 4$ and outer boundary $C$. Then there exists a vertex $v \in V(C)$ such that $B - v$ is a solid TB of order $|B| - 1$, unless $B\cong B_{15}^{(n)}$ for even $n\geq 6$ (see Figure \ref{fig:H4free}).
\end{pro}
\begin{proof} 
The proof proceeds by considering two separate cases.

\noindent{\bf Case 1.} $B$ is a near-triangulation.

If $C$ is chordless, then $B-v$ is a solid TB of order $|B|-1$ for each $v\in V(C)$. Otherwise, each chord partitions $B$ into two solid TBs. Choose a chord $ab$ of $C$ such that the smaller solid TB, say $B'$, has minimum order. Assume that the outer boundary of $B'$ is $C'$. Choose $v\in V(C')-\{a,b\}$. By the chord selection criterion, $v$ is incident to no chord of $C$. Hence, $B-v$ is a solid TB.

\noindent{\bf Case 2.} $B$ is not a near-triangulation.

 Let $\mathcal{F}$ be the set of holes in $B$. 
 Since $B$ is a solid TB, for each $F \in \mathcal{F}$, $|V(\partial(F))|\geq 4$ and $|V(\partial(F))\cap V(C)|\leq 1$.
 
% Then $vw$ belongs to a 3-face in $B$, without loss of generality, assume $vwzv$ is a 3-face containing edge $vw$. Then the 3-face $vwzv$ and any other 3-face adjacent to one edge outside $P_1$ in $F$ are not triangle-connected in $B$, a contradiction.

\noindent{\bf Case 2.1.} For each face $F\in \mathcal{F}$, $|V(\partial(F))\cap V(C)|=0$.

We first introduce a useful fact that will be frequently utilized in the following proof. 
Note that the following result holds since all faces incident to $v \in V(C)$ are 3-faces.

\textbf{Fact 1}:  For each $v\in V(C)$, $B[N(v)]$ is a path $P_k$, where $k=|N(v)|$.

Since $B$ is $H_4$-free, Fact 1 implies $2 \le |N(v)|\le 4$ for each $v\in C$. If some $v \in V(C)$ has $d_B(v) = 2$, say $N(v) = \{u, w\}$, then $uvwu$ is a 3-face by Fact 1. Then $B - \{v\}$ is a solid TB, so we can assume $3 \leq d_B(v) \leq 4$ for all $v \in V(C)$.

If $d_B(v) = 3$ for all $v \in V(C)$, let $C = v_1v_2v_3 \cdots v_{|C|}v_1$, and let $u$ be the third neighbor of $v_1$ (distinct from $v_2$ and $v_{|C|}$). Because $B[N[v_1]]$ is a 3-fan, $v_{|C|}uv_2$ is a path, so $u$ is adjacent to $v_2$. Inductively, since $B[N[v_{i-1}]]$ is a 3-fan for each $i \geq 3$, $u$ is adjacent to $v_i$. Thus, $B$ is the wheel graph $W_{|C|}$, implying $B$ is a near-triangulation, a contradiction. Therefore, $C$ must contain a vertex of degree 4.

We claim that $d_B(v)=4$ for each $v\in V(C)$.
Suppose, for contradiction, that there exist adjacent vertices \( u, v \in V(C) \) with \( d_B(u) = 3 \) and \( d_B(v) = 4 \). By Fact 1, \( B[N(v)] \) is a path. 
Let \( u, x_1, x_2, x_3 \) denote the neighbors of \( v \) listed in counter-clockwise order around \( v \), and let \( y_1, x_1, v \) be the neighbors of \( u \) listed in counter-clockwise order around \( u \). 
If \( y_1 = x_3 \), then \( x_1 \) and \( x_3 \) are both neighbors of \( u \), and  \( x_1x_3 \in E(B) \).
Since $d_B(x_3)\leq 4$, it follows that $x_1x_2x_3x_1$ is a $3$-face of $B$.
Hence, \( B \) is a near-triangulation, a contradiction. 
So we assume \( y_1 \neq x_3 \). Since no inner non-3-face intersects \( C \) at any vertex, the edge \( ux_1 \) lies on two 3-faces: \( vux_1v \) and \( uy_1x_1u \). 
Observe that \( y_1 \in V(C) \) and \( d_B(y_1) \geq 3 \), so \( y_1x_1 \notin E(C) \). Similarly, \( y_1x_1 \) lies on two 3-faces: \( uy_1x_1u \) and \( y_1y_2x_1y_1 \), where \( y_2 \in V(B) \). If \( y_2 \neq x_2 \), then \( y_1y_2x_1y_1 \cup \Theta_{vx_1} \) is an \( H_4 \), a contradiction. Thus, \( y_2 = x_2 \). Now, consider \( x_3 \in V(C) \), and let \( v, v' \) be its neighbors in \( C \). If \( y_1 = v' \), then \( B \) becomes a near-triangulation, a contradiction. Otherwise, \( y_1 \neq v' \), and \( x_2x_3 \) is belongs to two $3$-faces of $B$. Consequently, \( y_1x_1x_2y_1 \cup \Theta_{x_2x_3} \) forms a \( H_4\), yielding the final contradiction.
Hence, every vertex in \( V(C) \) has degree \( 4 \) in \( B \).

Let \( v_{|C|}, x_1, x_2, v_2 \) be the neighbors of \( v_1 \) in counter-clockwise order. By Fact 1, \( v_{|C|}x_1x_2v_2 \) is a path. Similarly, for each \( v_i \in C \), let \( v_{i-1}, x_{i-1}, x_i, v_{i+1} \) be its neighbors in counter-clockwise order (indices modulo \( |C| \)). By Fact 1, \( v_{i-1}x_ix_{i+1}v_{i+1} \) is a path, so \( x_i \) has neighbors \( x_{i-1}, v_i, v_{i-1}, x_{i+1} \). 
Now, suppose a vertex lies inside the cycle \( C'= x_0x_1 \cdots x_{|C|-1}x_0 \). Since \( B \) is connected, some \( x_i \) must have degree at least 5. Let \( y \) be its fifth neighbor. The edge \( x_i y \) lies in a 3-face of \( B \), forcing \( B[N[x_i]] \) to contain \( H_4 \), a contradiction. Thus, \( V(B) = V(C)\cup V(C') \) and $E(B)=E(C)\cup E(C')\cup\{v_ix_i,v_ix_{i+1}:i\in[|C|]\}$. 
In fact, the solid TB \( B \) is isomorphic to \( B_{15}^{(n)} \) for even $n\geq 6$, as shown in Figure \ref{fig:H4free}.

\noindent{\bf Case 2.2.} There exists a face $F\in \mathcal{F}$ such that $|V(\partial(F))\cap V(C)|=1$.

Assume that $V(\partial(F))\cap V(C)=\{v\}$. Suppose $v$ is incident to $a$ faces $F_1, \ldots, F_a$ (in counter-clockwise order) that are not 3-faces. Since each edge of $B$ lies in at least one 3-face, there is a 3-face between any two adjacent faces $F_i$ and $F_{i+1}$. Because $B$ is $H_4$-free,  $B[N[v]]$ is a friendship graph. Now we show that $B-v$ is a solid TB. Let $F'$ and $F''$ be two 3-faces in $B-v$, and let $\mathcal{P}$ be the set of connecting sequences between them in $B$. If some alternating sequence $P \in \mathcal{P}$ avoids 3-faces containing $v$, then $F' \sim F''$ in $B-v$. Otherwise, every $P \in \mathcal{P}$ includes a 3-face $F$ containing $v$, and thus contains a sub-alternating sequence $F^*eFeF^*$, where  $F^*$ is a 3-face with $F^*\cap F=\{e\}$. Removing all such sub-alternating sequences from $P$ yields a residual alternating sequence $P'$ connecting $F'$ and $F''$ within $B-v$. Therefore, $F'$ and $F''$ are triangular-connected in $B-v$, implying $B-v$ is a solid TB.
\end{proof}
 
Subsequently, we shall characterize all \(H_4\)-free solid TBs. This will be accomplished by initially analyzing the minimal configurations and then, in accordance with Proposition \ref{sec-1-pro-0}, systematically extending them through the addition of vertices. These \( H_4 \)-free solid TBs are exhibited in 
Figure~\ref{fig:H4free}.% It should be particularly noted that the order of each $B_i^{(n)}$ is $n$ for $i\in\{11,12,13,14,15\}$.

%%\begin{graph}%%%%%%%%%%%%%%%%%%%%%%%%%%%%%%%%%%%%%%%%%%%%%%%%%%
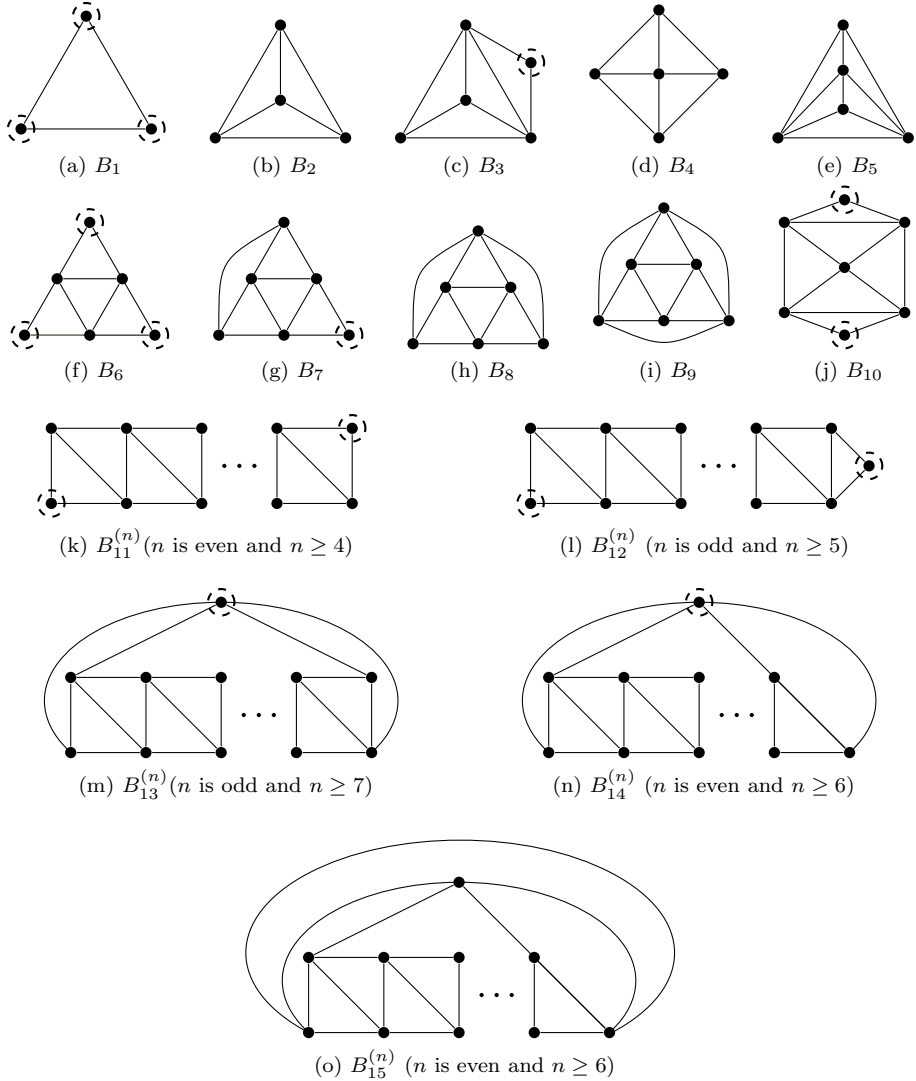
\begin{figure}[ht]
     \centering  
     \subfloat[{$B_1$}]
    {\begin{tikzpicture}
    \pgfmathparse{sqrt(3)}
    [inner sep=0.8pt]	
                \node[circle, fill, inner sep=1.5pt](v0) at (0,0)[]{};
               \node[circle, fill, inner sep=1.5pt](v1) at (\pgfmathresult,0)[]{};
                \node[circle, fill, inner sep=1.5pt](v2) at (\pgfmathresult/2,3/2)[]{};
                 \draw[dashed, thick] (v0) circle (5pt) ;
                \draw[dashed, thick] (v1) circle (5pt) ;
                \draw[dashed, thick] (v2) circle (5pt) ;
                \draw[-] (v0) -- (v1);
                \draw[-] (v1) -- (v2);
                \draw[-] (v2) -- (v0);
            \end{tikzpicture}
    }
    \hspace{1em} \subfloat[{$B_2$}]
   {\begin{tikzpicture}
    \pgfmathparse{sqrt(3)}
    [inner sep=0.1mm]	
               \node[circle, fill, inner sep=1.5pt](v0) at (0,0)[]{};
               \node[circle, fill, inner sep=1.5pt](v1) at (\pgfmathresult,0)[]{};
                \node[circle, fill, inner sep=1.5pt](v2) at (\pgfmathresult/2,3/2)[]{};
                \node[circle, fill, inner sep=1.5pt](v3) at (\pgfmathresult/2,1/2)[]{};
                 
                \draw[-] (v0) -- (v1);
                \draw[-] (v1) -- (v2);
                \draw[-] (v2) -- (v0);
                \draw[-] (v3) -- (v1);
                \draw[-] (v3) -- (v2);
                \draw[-] (v3) -- (v0);
            \end{tikzpicture}
    } \hspace{1em}  \subfloat[{$B_3$}]
   {\begin{tikzpicture}
    \pgfmathparse{sqrt(3)}
    [inner sep=0.1mm]	
               \node[circle, fill, inner sep=1.5pt](v0) at (0,0)[]{};
               \node[circle, fill, inner sep=1.5pt](v1) at (\pgfmathresult,0)[]{};
                \node[circle, fill, inner sep=1.5pt](v2) at (\pgfmathresult/2,3/2)[]{};
                \node[circle, fill, inner sep=1.5pt](v3) at (\pgfmathresult/2,1/2)[]{};
                 \node[circle, fill, inner sep=1.5pt](v4) at (\pgfmathresult,1)[]{};
                 \draw[dashed, thick] (v4) circle (5pt) ;
                \draw[-] (v0) -- (v1);
                \draw[-] (v1) -- (v2);
                \draw[-] (v2) -- (v0);
                \draw[-] (v3) -- (v1);
                \draw[-] (v3) -- (v2);
                \draw[-] (v3) -- (v0);
                \draw[-] (v4) -- (v1);
                \draw[-] (v4) -- (v2);
            \end{tikzpicture}
    } \hspace{1em} \subfloat[{$B_4$}]
       {\begin{tikzpicture}
    \pgfmathparse{sqrt(3)}
    [inner sep=0.1mm]	
               \node[circle, fill, inner sep=1.5pt](v0) at (0,0)[]{};
               \node[circle, fill, inner sep=1.5pt](v1) at (-0.85,0)[]{};
                \node[circle, fill, inner sep=1.5pt](v2) at (0,-0.85)[]{};
                \node[circle, fill, inner sep=1.5pt](v3) at (0.85,0)[]{};
                 \node[circle, fill, inner sep=1.5pt](v4) at (0,0.85)[]{};
               
                \draw[-] (v1) -- (v2);
                \draw[-] (v2) -- (v3);
                \draw[-] (v3) -- (v4);
                \draw[-] (v4) -- (v1);
                \draw[-] (v0) -- (v1);
                \draw[-] (v0) -- (v2);
                \draw[-] (v0) -- (v3);
                \draw[-] (v0) -- (v4);
            \end{tikzpicture}
    }
    \hspace{1em} \subfloat[{$B_5$}]
    {\begin{tikzpicture}
    \pgfmathparse{sqrt(3)}
    [inner sep=0.1mm]	
               \node[circle, fill, inner sep=1.5pt](v0) at (0,0)[]{};
               \node[circle, fill, inner sep=1.5pt](v1) at (\pgfmathresult,0)[]{};
                \node[circle, fill, inner sep=1.5pt](v2) at (\pgfmathresult/2,3/2)[]{};
                \node[circle, fill, inner sep=1.5pt](v3) at (\pgfmathresult/2,0.9)[]{};
                \node[circle, fill, inner sep=1.5pt](v4) at (\pgfmathresult/2,0.38)[]{};
                 
                \draw[-] (v0) -- (v1) -- (v2) -- (v0);
                \draw[-] (v2) -- (v3) -- (v4);
                \draw[-] (v0) -- (v3) -- (v1);
                \draw[-] (v0) -- (v4) -- (v1);
            \end{tikzpicture}
    } \\
       \subfloat[{$B_6$}]{\begin{tikzpicture} 
    \pgfmathparse{sqrt(3)}
    [inner sep=0.1mm]	
                \node[circle, fill, inner sep=1.5pt](v0) at (0,0)[]{};
               \node[circle, fill, inner sep=1.5pt](v1) at (\pgfmathresult,0)[]{};
                \node[circle, fill, inner sep=1.5pt](v2) at (\pgfmathresult/2,3/2)[]{};
                \node[circle, fill, inner sep=1.5pt](v3) at (\pgfmathresult/2,0)[]{};
               \node[circle, fill, inner sep=1.5pt](v4) at (1.2990381,3/4)[]{};
                \node[circle, fill, inner sep=1.5pt](v5) at (\pgfmathresult/4,3/4)[]{};
                \draw[dashed, thick] (v0) circle (5pt) ;
                \draw[dashed, thick] (v1) circle (5pt) ;
                \draw[dashed, thick] (v2) circle (5pt) ;
                \draw[-] (v0) -- (v1);
                \fill[yellow!40] (v0) -- (v3) -- (v5) -- cycle;
                \draw[-] (v1) -- (v2);
                \draw[-] (v2) -- (v0);
                \draw[-] (v3) -- (v4);
                \draw[-] (v4) -- (v5);
                \draw[-] (v5) -- (v3);
            \end{tikzpicture}
    } \hspace{1em}
    \subfloat[{$B_7$}]
    {\begin{tikzpicture}
    \pgfmathparse{sqrt(3)}
    [inner sep=0.1mm]	
                \node[circle, fill, inner sep=1.5pt](v0) at (0,0)[]{};
               \node[circle, fill, inner sep=1.5pt](v1) at (\pgfmathresult,0)[]{};
                \node[circle, fill, inner sep=1.5pt](v2) at (\pgfmathresult/2,3/2)[]{};
                \node[circle, fill, inner sep=1.5pt](v3) at (\pgfmathresult/2,0)[]{};
               \node[circle, fill, inner sep=1.5pt](v4) at (1.2990381,3/4)[]{};
                \node[circle, fill, inner sep=1.5pt](v5) at (\pgfmathresult/4,3/4)[]{};
                \draw[dashed, thick] (v1) circle (5pt) ;
                \draw[-] (v0) -- (v1);
                \draw[-] (v1) -- (v2);
                \draw[-] (v2) -- (v0);
                \draw[-] (v3) -- (v4);
                \draw[-] (v4) -- (v5);
                \draw[-] (v5) -- (v3);
                 % \draw[dashed] (v0) to[out=-45,in=-135] node[midway,right=5pt]{}(v1) ;
                   \draw[-] (v0) .. controls (0,1) .. (v2);
            \end{tikzpicture}
    }   \hspace{1em} 
    \subfloat[{$B_8$}]
    {\begin{tikzpicture}
    \pgfmathparse{sqrt(3)}
    [inner sep=0.1mm]	
                \node[circle, fill, inner sep=1.5pt](v0) at (0,0)[]{};
               \node[circle, fill, inner sep=1.5pt](v1) at (\pgfmathresult,0)[]{};
                \node[circle, fill, inner sep=1.5pt](v2) at (\pgfmathresult/2,3/2)[]{};
                \node[circle, fill, inner sep=1.5pt](v3) at (\pgfmathresult/2,0)[]{};
               \node[circle, fill, inner sep=1.5pt](v4) at (1.2990381,3/4)[]{};
                \node[circle, fill, inner sep=1.5pt](v5) at (\pgfmathresult/4,3/4)[]{};

                \draw[-] (v0) -- (v1);
                \draw[-] (v1) -- (v2);
                \draw[-] (v2) -- (v0);
                \draw[-] (v3) -- (v4);
                \draw[-] (v4) -- (v5);
                \draw[-] (v5) -- (v3);
                 % \draw[dashed] (v0) to[out=-45,in=-135] node[midway,right=5pt]{}(v1) ;
                   \draw[-] (v0) .. controls (0,1) .. (v2);
                  \draw[-] (v1) .. controls (\pgfmathresult,1) .. (v2);
            \end{tikzpicture}
    }  \hspace{1em} 
    \subfloat[{$B_9$}]
    {\begin{tikzpicture}
    \pgfmathparse{sqrt(3)}	
                \node[circle, fill, inner sep=1.5pt](v0) at (0,-3/8)[]{};
               \node[circle, fill, inner sep=1.5pt](v1) at (\pgfmathresult,-3/8)[]{};
                \node[circle, fill, inner sep=1.5pt](v2) at (\pgfmathresult/2,3/2-3/8)[]{};
                \node[circle, fill, inner sep=1.5pt](v3) at (\pgfmathresult/2,-3/8)[]{};
               \node[circle, fill, inner sep=1.5pt](v4) at (1.2990381,3/4-3/8)[]{};
                \node[circle, fill, inner sep=1.5pt](v5) at (\pgfmathresult/4,3/4-3/8)[]{};

                \draw[-] (v0) -- (v1);
                \draw[-] (v1) -- (v2);
                \draw[-] (v2) -- (v0);
                \draw[-] (v3) -- (v4);
                \draw[-] (v4) -- (v5);
                \draw[-] (v5) -- (v3);
                 % \draw[dashed] (v0) to[out=-45,in=-135] node[midway,right=5pt]{}(v1) ;
                   \draw[-] (v0) .. controls (0,5/8) .. (v2);
                  \draw[-] (v1) .. controls (\pgfmathresult,5/8) .. (v2);
                    \draw[-] (v0) .. controls (\pgfmathresult/2,-6/8) .. (v1);
            \end{tikzpicture}
    } \hspace{1em} 
    \subfloat[{$B_{10}$}]
    {\begin{tikzpicture}
    \pgfmathparse{sqrt(3)}	
                \node[circle, fill, inner sep=1.5pt](v0) at (0,0)[]{};
               \node[circle, fill, inner sep=1.5pt](v1) at (1.6,0)[]{};
                \node[circle, fill, inner sep=1.5pt](v2) at (1.6,1.2)[]{};
                \node[circle, fill, inner sep=1.5pt](v3) at (0,1.2)[]{};
               \node[circle, fill, inner sep=1.5pt](v4) at (0.8,0.6)[]{};
                \node[circle, fill, inner sep=1.5pt](v5) at (0.8,1.5)[]{};
                \node[circle, fill, inner sep=1.5pt](v6) at (0.8,-0.3)[]{};
              \draw[dashed, thick] (v5) circle (5pt);
              \draw[dashed, thick] (v6) circle (5pt) ;
                \draw[-] (v0) -- (v1);
                \draw[-] (v1) -- (v2);
                \draw[-] (v2) -- (v3);
                \draw[-] (v3) -- (v0);
                \draw[-] (v4) -- (v0);
                \draw[-] (v4) -- (v1);
                \draw[-] (v4) -- (v2);
                \draw[-] (v4) -- (v3);
                \draw[-] (v5) -- (v2);
                \draw[-] (v5) -- (v3);
                \draw[-] (v6) -- (v0);
                \draw[-] (v6) -- (v1);              
            \end{tikzpicture}
    }\\
    \subfloat[{$B_{11}^{(n)}$($n$ is even and $n\geq 4$)}]
    {\begin{tikzpicture}
    \pgfmathparse{sqrt(3)}	
               \node[circle, fill, inner sep=1.5pt](v1) at (0,0)[]{};
                \node[circle, fill, inner sep=1.5pt](v2) at (1,0)[]{};
                \node[circle, fill, inner sep=1.5pt](v3) at (1,1)[]{};
               \node[circle, fill, inner sep=1.5pt](v4) at (0,1)[]{};
                \node[circle, fill, inner sep=1.5pt](u1) at (2,0)[]{};
               \node[circle, fill, inner sep=1.5pt](u2) at (2,1)[]{};
            
                \node[circle, fill, inner sep=1.5pt](v5) at (3,0)[]{};
                \node[circle, fill, inner sep=1.5pt](v6) at (4,0)[]{};  
                \node[circle, fill, inner sep=1.5pt](v7) at (4,1)[]{};
                \node[circle, fill, inner sep=1.5pt](v8) at (3,1)[]{};
              \draw[dashed, thick] (v1) circle (5pt) ;
              \draw[dashed, thick] (v7) circle (5pt) ;
            \fill (2.3,0.5) circle (0.8pt);
            \fill (2.5,0.5) circle (0.8pt); 
            \fill (2.7,0.5) circle (0.8pt);
                % \draw[-] (v0) -- (v4);
                % \draw[-] (v0) -- (v1);
                \draw[-] (v1) -- (v2);
                \draw[-] (v2) -- (v3);
                \draw[-] (v3) -- (v4);
                \draw[-] (v4) -- (v1);
                \draw[-] (v5) -- (v6);
                \draw[-] (v6) -- (v7);
                \draw[-] (v7) -- (v8);
                \draw[-] (v8) -- (v5);
                % \draw[-] (v9) -- (v6);
                % \draw[-] (v9) -- (v7);     
                \draw[-] (v8) -- (v6);  
                \draw[-] (v4) -- (v2); 
             \draw (v2) -- (u1) -- (u2) -- (v3)--(u1);
            \end{tikzpicture}
    }\hspace{5em} 
    \subfloat[{$B_{12}^{(n)}$ ($n$ is odd and $n\geq 5$)}]
    {\begin{tikzpicture}
    \pgfmathparse{sqrt(3)}	
               \node[circle, fill, inner sep=1.5pt](v1) at (0,0)[]{};
                \node[circle, fill, inner sep=1.5pt](v2) at (1,0)[]{};
                \node[circle, fill, inner sep=1.5pt](v3) at (1,1)[]{};
               \node[circle, fill, inner sep=1.5pt](v4) at (0,1)[]{};
                 \node[circle, fill, inner sep=1.5pt](u1) at (2,0)[]{};
               \node[circle, fill, inner sep=1.5pt](u2) at (2,1)[]{};
                \node[circle, fill, inner sep=1.5pt](v5) at (3,0)[]{};
                \node[circle, fill, inner sep=1.5pt](v6) at (4,0)[]{};  
                \node[circle, fill, inner sep=1.5pt](v7) at (4,1)[]{};
                \node[circle, fill, inner sep=1.5pt](v8) at (3,1)[]{};
                \node[circle, fill, inner sep=1.5pt](v9) at (4.5,0.5)[]{};
            \fill (2.3,0.5) circle (0.8pt); % 模拟省略号的一个小圆点
            \fill (2.5,0.5) circle (0.8pt); % 再加一个，表示更多的省
             \fill (2.7,0.5) circle (0.8pt);
             \draw[dashed, thick] (v0) circle (5pt) ;
             \draw[dashed, thick] (v9) circle (5pt) ;
                \draw[-] (v9) -- (v6);
                \draw[-] (v9) -- (v7);
                \draw[-] (v1) -- (v2);
                \draw[-] (v2) -- (v3);
                \draw[-] (v3) -- (v4);
                \draw[-] (v4) -- (v1);
                \draw[-] (v5) -- (v6);
                \draw[-] (v6) -- (v7);
                \draw[-] (v7) -- (v8);
                \draw[-] (v8) -- (v5);   
                \draw[-] (v8) -- (v6);  
                \draw[-] (v4) -- (v2); 
                \draw (v2) -- (u1) -- (u2) -- (v3)--(u1);
            \end{tikzpicture}
    } \\
    \subfloat[{$B_{13}^{(n)}$($n$ is odd and $n\geq 7$)}]
    {\begin{tikzpicture}
    \pgfmathparse{sqrt(3)}	
                \node[circle, fill, inner sep=1.5pt](v0) at (2,2)[]{};
               \node[circle, fill, inner sep=1.5pt](v1) at (0,0)[]{};
                \node[circle, fill, inner sep=1.5pt](v2) at (1,0)[]{};
                \node[circle, fill, inner sep=1.5pt](v3) at (1,1)[]{};
               \node[circle, fill, inner sep=1.5pt](v4) at (0,1)[]{};
                \node[circle, fill, inner sep=1.5pt](u1) at (2,0)[]{};
               \node[circle, fill, inner sep=1.5pt](u2) at (2,1)[]{};
            
                \node[circle, fill, inner sep=1.5pt](v5) at (3,0)[]{};
                \node[circle, fill, inner sep=1.5pt](v6) at (4,0)[]{};  
                \node[circle, fill, inner sep=1.5pt](v7) at (4,1)[]{};
                \node[circle, fill, inner sep=1.5pt](v8) at (3,1)[]{};
            \fill (2.3,0.5) circle (0.8pt);
            \fill (2.5,0.5) circle (0.8pt); 
            \fill (2.7,0.5) circle (0.8pt);
            \draw[dashed, thick] (v0) circle (5pt) ;
                \draw[-] (v0) -- (v4);
                \draw[-] (v1) -- (v2);
                \draw[-] (v2) -- (v3);
                \draw[-] (v3) -- (v4);
                \draw[-] (v4) -- (v1);
                \draw[-] (v5) -- (v6);
                \draw[-] (v6) -- (v7);
                \draw[-] (v7) -- (v8);
                \draw[-] (v8) -- (v5);
                \draw[-] (v0) -- (v7);     
                \draw[-] (v8) -- (v6);  
                \draw[-] (v4) -- (v2); 
             \draw (v2) -- (u1) -- (u2) -- (v3)--(u1);
               \draw (2,2)arc (90:210:2.35cm and 1.31cm);
               \draw (2,2)arc (90:-30:2.35cm and 1.31cm);
               % \draw[-] (v4) .. controls (-0.6,-0.3) .. (v0);
             % \draw[-] (v7) .. controls (4.6,-0.3) .. (v0);
            \end{tikzpicture}
    }\hspace{4em} 
    \subfloat[{$B_{14}^{(n)}$ ($n$ is even and $n\geq 6$)}]
    {\begin{tikzpicture}
    \pgfmathparse{sqrt(3)}	
                \node[circle, fill, inner sep=1.5pt](v0) at (2,2)[]{};
               \node[circle, fill, inner sep=1.5pt](v1) at (0,0)[]{};
                \node[circle, fill, inner sep=1.5pt](v2) at (1,0)[]{};
                \node[circle, fill, inner sep=1.5pt](v3) at (1,1)[]{};
               \node[circle, fill, inner sep=1.5pt](v4) at (0,1)[]{};
                 \node[circle, fill, inner sep=1.5pt](u1) at (2,0)[]{};
               \node[circle, fill, inner sep=1.5pt](u2) at (2,1)[]{};
                \node[circle, fill, inner sep=1.5pt](v5) at (3,0)[]{};
                \node[circle, fill, inner sep=1.5pt](v6) at (4,0)[]{};  
                \node[circle, fill, inner sep=1.5pt](v7) at (3,1)[]{};
            \fill (2.3,0.5) circle (0.8pt); % 模拟省略号的一个小圆点
            \fill (2.5,0.5) circle (0.8pt); % 再加一个，表示更多的省
             \fill (2.7,0.5) circle (0.8pt);
             \draw[dashed, thick] (v0) circle (5pt) ;
                \draw[-] (v0) -- (v4);
                \draw[-] (v1) -- (v2);
                \draw[-] (v2) -- (v3);
                \draw[-] (v3) -- (v4);
                \draw[-] (v4) -- (v1);
                \draw[-] (v5) -- (v6);
                \draw[-] (v6) -- (v0);
                \draw[-] (v7) -- (v5);   
                \draw[-] (v8) -- (v6);  
                \draw[-] (v4) -- (v2); 
                \draw (v2) -- (u1) -- (u2) -- (v3)--(u1);
              \draw (2,2)arc (90:210:2.35cm and 1.31cm);
               \draw (2,2)arc (90:-30:2.35cm and 1.31cm);
            \end{tikzpicture}
    } \\ \subfloat[{$B_{15}^{(n)}$ ($n$ is even and $n\geq 6$)}]
    {\begin{tikzpicture}
    \pgfmathparse{sqrt(3)}	
                \node[circle, fill, inner sep=1.5pt](v0) at (2,2)[]{};
               \node[circle, fill, inner sep=1.5pt](v1) at (0,0)[]{};
                \node[circle, fill, inner sep=1.5pt](v2) at (1,0)[]{};
                \node[circle, fill, inner sep=1.5pt](v3) at (1,1)[]{};
               \node[circle, fill, inner sep=1.5pt](v4) at (0,1)[]{};
                 \node[circle, fill, inner sep=1.5pt](u1) at (2,0)[]{};
               \node[circle, fill, inner sep=1.5pt](u2) at (2,1)[]{};
                \node[circle, fill, inner sep=1.5pt](v5) at (3,0)[]{};
                \node[circle, fill, inner sep=1.5pt](v6) at (4,0)[]{};  
                \node[circle, fill, inner sep=1.5pt](v7) at (3,1)[]{};
            \fill (2.3,0.5) circle (0.8pt); % 模拟省略号的一个小圆点
            \fill (2.5,0.5) circle (0.8pt); % 再加一个，表示更多的省
             \fill (2.7,0.5) circle (0.8pt);
                \draw[-] (v0) -- (v4);
                \draw[-] (v1) -- (v2);
                \draw[-] (v2) -- (v3);
                \draw[-] (v3) -- (v4);
                \draw[-] (v4) -- (v1);
                \draw[-] (v5) -- (v6);
                \draw[-] (v6) -- (v0);
                \draw[-] (v7) -- (v5);   
                \draw[-] (v8) -- (v6);  
                \draw[-] (v4) -- (v2); 
                \draw (v2) -- (u1) -- (u2) -- (v3)--(u1);
              \draw (2,2)arc (90:210:2.35cm and 1.31cm);
               \draw (2,2)arc (90:-30:2.35cm and 1.31cm);
              \draw (0,0)arc (225:-45:2.85cm and 1.5cm);
            \end{tikzpicture}
    } 
  \caption{\label{fig:H4free} $H_4$-free triangular blocks, with dashed circles indicating potential junction vertices. $|B_i^{(n)}|=n$  for $i\in\{11,12,13,14,15\}$.}
    % \caption{All triangular blocks that do not contain $K_1+(P_2\cup P_3)$ as a subgraph.}
    \end{figure}
%%\end{graph}%%%%%%%%%%%%%%%%%%%%%%%%%%%%%%%%%%%%%%%%%%%%%%%%%%

\begin{lem}\label{sec-1-lem-1}
Every $H_4$-free solid TB  is isomorphic to a configuration in Figure~\ref{fig:H4free}.
\end{lem}
\begin{proof}
All solid TBs of order at most 5 are $H_4$-free, as exemplified by configurations $B_1$ to $B_5$, $B_{11}^{(4)}$, and $B_{12}^{(5)}$.
Additionally, $B_{15}^{(n)}$ (see Figure~\ref{fig:H4free}) is $H_4$-free for even $n\geq 6$.
Hence, we assume $B$ is $H_4$-free solid TB of order $n\geq 6$ and $B$  is not isomorphic to $B_{15}^{(n)}$ for any even $n\geq 6$.

For $|B|=6$, since $B$ is not isomorphic to $B_{15}^{(6)}$, Proposition \ref{sec-1-pro-0} ensures that there is a vertex $v$ on the outer boundary of $B$ such that $B-v$ is a copy of $B_3$, $B_4$, $B_5$ or $B_{12}^{(5)}$ as a subgraph. If $B-v$ is a copy of $B_3$ or $B_5$, then $B$ contains an $H_4$ as a subgraph, which is impossible. If $B-v$ is a copy of $B_4$,  then $B$ is isomorphic to $B_7$, $B_8=B_{14}^{(6)}$ or $B_9$. If $B-v$ is a copy of  $B_{12}^{(5)}$, then $B$ is isomorphic to $B_6$, $B_7$, $B_8$ or $B_{11}^{(6)}$.

For $|B|=7$, Proposition \ref{sec-1-pro-0} ensures that $B-v$ is a copy of $B_6$, $B_7$, $B_8$, $B_9$, $B_{11}^{(6)}$ or $B_{15}^{(6)}$ (note $B_9\cong B^6_{15}$). If $B-v$ is a copy of $B_6,B_8$, $B_9$ or $B_{15}^{(6)}$, then $B$ contains an $H_4$, a contradiction. If $B-v$ is a copy of $B_7$, then $B$ is isomorphic to $B_{10}$.  If $B-v$ is a copy of $B_{11}^{(6)}$, then $B$ is isomorphic to $B_{12}^{(7)}$ or $B_{13}^{(7)}$.

For $|B|=8$, since $B$ is not isomorphic to $B_{15}^{(8)}$, there is a vertex $v$ on the outer boundary of $B$ such that $B-v$ is a copy of $B_{10}$, $B_{12}^{(7)}$, or $B_{13}^{(7)}$. If $B-v$ is a copy of $B_{10}$ or $B_{13}^{(7)}$, then $B$ contains an $H_4$, a contradiction. If $B-v$ is a copy of $B_{12}^{(7)}$, then $B$ is isomorphic to $B_{11}^{(8)}$ or $B_{14}^{(8)}$.

Now we assume that $|B|=n\geq 9$, since $B$ is not isomorphic to $B_{15}^{(n)}$ for any even $n\geq 10$, there is a vertex $v$  on the outer boundary of $B$ such that $B-v$ is a solid TB of order $n-1$. If $n$ is odd, then $B-v$ is a copy of $B_{11}^{(n-1)}$, $B_{14}^{(n-1)}$ or $B_{15}^{(n-1)}$. 
If $B-v$ is a copy of $B_{11}^{(n-1)}$, then $B$ is isomorphic to $B_{12}^{(n)}$ or $B_{13}^{(n)}$. If  $B$ is a copy of  $B_{14}^{(n-1)}$ or $B_{15}^{(n-1)}$, then $B$ contains an $H_4$, a contradiction. If $n$ is even, then $B-v$ is a copy of $B_{12}^{(n-1)}$ or $B_{13}^{(n-1)}$. If $B-v$ is a copy of $B_{12}^{(n-1)}$, then $B$ is isomorphic to $B_{11}^{(n)}$, $B_{14}^{(n)}$.
If $B-v$ is a copy of $B_{13}^{(n-1)}$, then $B$ contains an $H_4$, a contradiction.
\end{proof}

\begin{table}[htbp]
 {  \centering
    \begin{tabular} 
    {|c|c|c|c|c|c|c|c|c|c|c|c|c|} % 表格的列由竖线分隔，c表示居中对齐
  \hline % 第一行水平线（表头下面的线）
  % Case & (a) & (b) & (c) & (d) & (e) & (f) & (g) & (h) & (i) & (j) & (k) & (l) \\ % 
   Case & $B_1$ & $B_2$ & $B_3$ & $B_4$ & $B_5$  & $B_6$ & $B_7$ & $B_8$\\ %第一行内容
    \hline % 第二行水平线（第一行内容下面的线）
   $\Delta_{B_i}$ & 1 & 3  & 4 & 4 & 5 & 4 & 5 & 6 \\ % 第二行内容
  \hline % 第三行水平线（第二行内容下面的线）
  Triangle density & $1/3$ & $3/4$  & $4/5$ & $4/5$ & 1 & $2/3$ & $5/6$ & 1 \\ % 第三行内容
  \hline % 第四行水平线（第三行内容下面的线）
  \hline
  Case & $B_9$  &$B_{10}$ & $B_{11}^{(n)}$ & $B_{12}^{(n)}$ & $B_{13}^{(n)}$ & $B_{14}^{(n)}$ & $B_{15}^{(n)}$&\\
    \hline % 第五行水平线（第四行内容下面的线）
   $\Delta_{B_i}$ & 7&  6  & $n-2$ & $n-2$ & $n-1$ & $n-1$ & $n$  & \\ % 第五行内容
    \hline % 第五行水平线（第五行内容下面的线）
  Triangle density  &$7/6$ & $6/7$ & $(n-2)/n$ & $(n-2)/n$ & $(n-1)/n$ & $(n-1)/n$ & $1$&\\
  \hline % 最后一行水平线（表格底部的线，可选）
\end{tabular}
\caption{\label{tab:1}The triangle-densities of triangle-blocks in $\mathcal{B}$.}
}
\end{table}

Let $\mathcal{B}=\{B_i|1\le i\le 10\}\cup \{B^{(n)}_i|11\le i\le 15\}$. Table \ref{tab:1} summarizes the  triangle-densities for solid TBs in $\mathcal{B}$.
Let $D$ be an $H_4$-free TC composed of solid TBs. One can easily verify that $D$ is formed by connecting solid TBs in $\mathcal{B}$ via junction vertices  (see Figure \ref{fig:H4free}, where the junction vertices in each solid TB are marked by dots enclosed with dashed circles). 
Clearly, suppose, for contradiction, that there are two solid TBs $B_1$ and $B_2$ such that a non-junction vertex $v$ of $B_1$ is also a vertex in $B_2$. Let $F$ be an inner $3$-face of $B_2$ containing $v$. If $V(B_1)\cap V(\partial(F))=\{v\}$,
then $B_1\cup B_2$ contains an $H_4$ as a subgraph.
For the case where $|V(B_1)\cap V(\partial(F))|\geq 2$, we can similarly confirm that $B_1\cup B_2$ contains an $H_4$ through different combinations.
Hence, in both cases, $B_1\cup B_2$ contains an $H_4$, a contradiction.

It is worth noting that not all solid TBs in $\mathcal{B}$ admit such a junction vertex; for example, the solid TBs in 
$$\mathcal{B}_1=\{B_2,B_4,B_5,B_8,B_9,B_{15}^{(n)}\}$$ 
are excluded due to the $H_4$-free constraint of $D$. For the same reason, some solid TBs contain at most one junction vertex, such as solid TBs in 
$$\mathcal{B}_2=\{B_3, B_7, B_{13}^{(n)}, B_{14}^{(n)}\}.$$ 
The remaining solid TBs are denoted by 
$$\mathcal{B}_3=\{B_1,B_6,B_{10},B_{11}^{(n)},B_{12}^{(n)}\},$$ 
each containing at least two junction vertices.
Define a solid TB $B$ as $B_i$-{\em type} TB if $B\cong B_i$. 
Before proceeding the proof of Theorem \ref{thm1}, we establish a foundational lemma on triangle-densities of $D$.

\begin{lem}\label{Sec1-densityoftriangles}
Let $D$ be an $H_4$-free TC of order at least seven. If  $D\notin \{B_{15}^{(8)},B_{15}^{(10)}\}$, then $\rho(D)\le \frac{6|D|-12}{5|D|}$.
\end{lem}
\begin{proof}
If we replace each TB in $D$ by the corresponding solid TB, then the resulting TC is also $H_4$-free, and its triangle density does not decrease.
Hence, we can assume that $D$ is an $H_4$-free TC of order at least seven and is composed of solid TBs.

We proceed by induction on the number $k$ of solid TBs in $D$. 
For the base case $k=1$, since $|D| \ge 7$ and $D\notin \{B_{15}^{(8)},B_{15}^{(10)}\}$,  $\rho(D)\le \frac{6|D|-12}{5|D|}$ from Table \ref{tab:1}.
For the inductive step with $k\ge 2$, each  solid TB must belong to $\mathcal{B}_2\cup\mathcal{B}_3$ as they are connected through junction vertices.

\noindent{\bf Case 1.} There exists a solid TB of $D$, say $B$, containing exactly one junction vertex $v$.

Choose such a $B$ with $|B|$ minimum. It is easy to verify that $\Delta_B \leq |B| - 1$ since $B \in \mathcal{B}_2 \cup \mathcal{B}_3$. Let $D' = D - V(B - v)$. 
If $|D'|=3$, then both $D'$ and $B$ are $B_1$-type TBs, and hence 
$\rho(D)=2/5< \frac{6|D|-12}{5|D|}$.
Hence, we assume that  $|D'| \geq 4$. 
It is obvious that $D'$ is a TC and $|D| = |D'| + |B| - 1$.

If $|D'| \leq 6$, then $|D'| \in \{4, 5, 6\}$. If $|D'| = 6$, then $\Delta_{D'} \leq 5$, with equality holding only if $D' \cong B_7$ or $D'$ is the union of $B_1$ and $B_6$ by identifying three junction vertices.
If $|D'| = 5$, then $\Delta_{D'} \leq 4$, with equality holding only if $D' \cong B_3$. If $|D'| = 4$, then $D' \cong B_{11}^4$ and $\Delta_{D'} = 2$.
In all these cases, $\Delta_{D'} \leq |D'| - 1$. Therefore,
$$\Delta_D \leq \Delta_{D'} + (|B| - 1) \leq (|D'| - 1) + (|B| - 1) = |D'| + |B| - 2 = |D| - 1,$$
and
$$\rho(D) = \frac{\Delta_D}{|D|} \leq 1 - \frac{1}{|D|} \leq \frac{6|D| - 12}{5|D|},$$
the last inequality holds since $|D|\geq 7$.

If $|D'|\geq 7$, then by induction, $\rho(D')\leq \frac{6|D'|-12}{5|D'|}$.
Note that $\Delta_B\leq |B|-1$.
Hence, 
$$\Delta_{D}\leq \frac{6|D'|-12}{5}+(|B|-1)<\frac{6(|D'|+|B|-1)-12}{5}=\frac{6|D|-12}{5},$$ 
implying $\rho(D)<\frac{6|D|-12}{5|D|}$.

\noindent\textbf{Case 2.}  There is a solid TB of $D$, denoted $H$, containing precisely two junction vertices $u$ and $v$.

We can assume that each  solid TB in $D$ belongs to $\mathcal{B}_3$, for otherwise there is a solid TB with only one junction $v$, which has been discussed in Case 1.
Let $D'$ be obtained form $D$ by removing all edges in $H$, and then deleting $V(H)-\{u,v\}$. Obviously, $|D'|=|D|-|H|+2$.

\noindent{\em Case 2.1.} $D'$ is connected. 

If $|D'|\geq 7$, then by induction, $\frac{\Delta_{D'}}{|D'|}\leq \frac{6|D'|-12}{5|D'|}$. 
If $|D'|\leq 6$, then either $D'$ is a solid TB in $\{B_1,B_6,B_{11}^{(4)},B_{11}^{(6)},B_{12}^{(5)}\}$ or $D'$ is one of the following configurations (see Figure\ref{D'}):
\begin{enumerate}
    \item the union of two or three $B_1$-type TBs (see graphs (a) and (b) in Figure \ref{D'}),
    \item the union of $B^4_{11}$ and $B_1$ (see the graph (c) and (e) in Figure \ref{D'}), 
    \item the union of two $B^4_{11}$-type TBs (see the graph (f) in Figure \ref{D'}),
    \item the union of $B^5_{12}$ and $B_1$ (see the graph (d) in Figure \ref{D'}),
\end{enumerate}
In all cases, one can verify that $\rho(D')\leq\frac{6|D'|-12}{5|D'|}$.

\begin{figure}[ht]
	\centering
	\includegraphics[width=350pt]{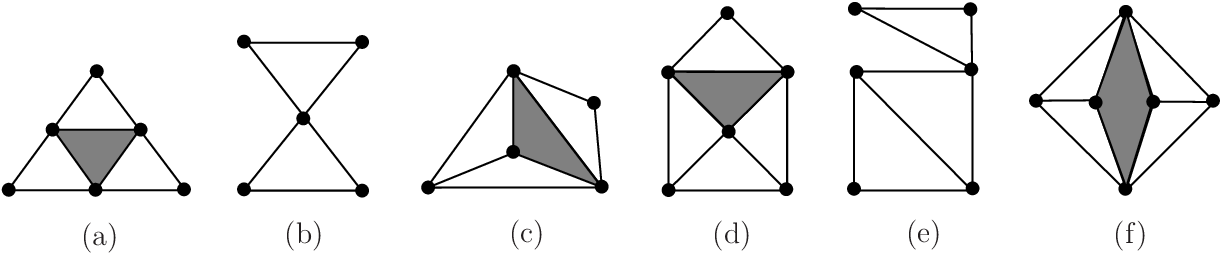}
	\caption{\label{D'} all possible configurations of $D'$ when $|D'|\leq 6$, where the gray area denotes the hole.}
\end{figure}

If $H\ncong B_{10}$, then $\Delta_H\leq |H|-2$.
Thus,
   $$\rho(D)=\frac{\Delta_{D'}+\Delta_H}{|D|}\leq \frac{\frac{6|D'|-12}{5}+|H|-2}{|D|}=\frac{6|D|-12-(|H|-2)}{5|D|}< \frac{6|D|-12}{5|D|}.$$    
If $H\cong B_{10}$, then $$\rho(D)=\frac{\Delta_{D'}+6}{|D'|+5}\leq \frac{\frac{6|D'|-12}{5}+6}{|D|}=\frac{6(|D'|+5)-12}{5|D|}=\frac{6(|D|)-12}{5|D|}.$$

\noindent{\em Case 2.2.} $D'$ is not connected.

Let $D''$ be obtained from $D$ by deleting $E(H)$ and then contracting $V(H)$ to a single vertex $w$.
Then $D''$ is connected and $w$ is a cut-vertex of $D''$.
Note that $w$ is a junction vertex of degree two in each solid TB.
Hence, $D''$ remains a $H_4$-free TC.
If $|D''|\geq 7$, then by induction,  $\frac{\Delta_{D''}}{|D''|}\leq \frac{6|D''|-12}{5|D''|}$.
If $|D''|\leq 6$, then since $|D'|=|D''|+1$ and $D'$ contains at least two components, it follows that $D''$ is a configuration of (c) or (e) in Figure \ref{D'}, which implies that $\frac{\Delta_{D''}}{|D''|}\leq \frac{6|D''|-12}{5|D''|}$.
Therefore, $\frac{\Delta_{D''}}{|D''|}\leq \frac{6|D''|-12}{5|D''|}$.
Since $\Delta_H\leq |H|-1$ as $H\in \mathcal{B}_3$,
it follows that 
$$\Delta_D\leq \frac{6|D''|-12}{5}+(|H|-1)=\frac{6|D|-12+(1-|H|)}{5}<\frac{6|D|-12}{5},$$
implying $\rho(D)<\frac{6|D|-12}{5|D|}$.
\smallskip

\noindent\textbf{Case 3.} Each solid TB of $D$ possesses three junction vertices.

It is clear that each solid TB is either a copy of $B_1$ of a copy of $B_6$, and all three vertices of degree $2$ in the TB are junction vertices.
Since each solid TB is $2$-connected, it follows that $D$ is $2$-connected.
Hence, the boundary of each face in $D$ is a cycle.
Thus, the inner faces of $D$ can be partitioned into two classes: the inner $3$-faces within each solid TB and holes.
Note that if a hole is a $3$-face in $D$, then it is not a $3$-face in $G$.

Arbitrarily choose a hole $F$ (denoting its facial cycle by $C$), then each edge of $C$ is contained in exactly one solid TB of $D$. Furthermore, if a solid TB $B$ intersects $C$ in at least one edge, then: 
\begin{enumerate}
    \item [(A)] if $B$ is isomorphic to $B_1$, then $|E(B)\cap E(C)|=1$, and
    \item [(B)] if $B$ is isomorphic to $B_6$, then $|E(B)\cap E(C)|=2$. Moreover, $B\cap C$ is a $3$-path whose endpoints are two of the junction vertices of $B$.
\end{enumerate}

\noindent{\em Case 3.1.} $D$ has a hole $F$ whose facial cycle is $C=C_3$.

If each edge of $C$ belongs to a TB that is isomorphic to $B_1$ (say $V(C)=\{a_1,a_2,a_3\}$, and $a_ia_{i+1}$ is an edge of the $B_1$-type TB $T_i$), then by (A), $T_1,T_2$ and $T_3$ are pairwise distinct.
Since $D$ is $H_4$-free, the junction vertex $a_i$ is not in any other TBs.
Let $D'$ be a plane graph obtained from $D$ by deleting $a_1,a_2,a_3$.
Then $D'$ is an $H_4$-free TC  obtained from $D$  by removing three solid TBs $T_1,T_2$ and $T_3$.
Additionally, since $D$ is $2$-connected and  $a_1,a_2,a_3$ are not junction vertices, it follows that $D'$ is connected.
If $|D'|\geq 7$, then by induction, $\rho(D')\leq \frac{6|D'|-12}{5|D'|}$.
If $3\leq |D'|\leq 6$, then $D'$ is either a configuration of (a) or (b) in Figure \ref{D'}, or a copy of $B_6$.
In either cases, we can verify that $\rho(D')\leq \frac{6|D'|-12}{5|D'|}$.
Therefore, $\rho(D')\leq \frac{6|D'|-12}{5|D'|}$ holds.
Since $\Delta_D-3=\Delta_{D'}$ and $|D|-3=|D'|$, it follows that
$$\rho(D)\leq \frac{\Delta_{D'}+3}{|D'|+3}\leq\frac{\frac{6|D'|-12}{5}+3}{|D'|+3}=\frac{6(|D'|+3)-15}{5(|D'|+3)}<\frac{6|D|-12}{5|D|}.$$

If there is an edge of $C$ belongs to a block $T$ that is isomorphic to $B_6$, then by (B), $|E(C)\cap E(T)|=2$ and $B\cap C$ is a $2$-path whose endpoints are two of the junction vertices of $T$.
Without loss of generality, assume that $a_1,a_2,b$ are junction vertices of $T$ and $E(C)\cap E(T)=\{a_1a_3,a_2a_3\}$.
We further assume that $a_1a_2$ belongs to the block $T'$.
By (A), $T'$ is ao $B_1$-type TB (assume that $V(T')=\{a_1,a_2,c\}$).
Let $D'$ is a plane graph obtained from $D$ by deleting vertices  $(V(T\cup T')-\{b,c\})$ and removing $E(T)\cup E(T')$.
Since $D$ is $H_4$-free, the junction vertices $a_1$ and $a_2$ are not in any other TBs. 
Since $D$ is $2$-connected and $b,c$ are the only two junction vertices in $T\cup T'$ connecting other solid TBs, it follows that $D'$ is connected.
If $|D'|\geq 7$, then by induction, $\rho(D')\leq \frac{6|D'|-12}{5|D'|}$.
If $3\leq |D'|\leq 6$, then $D'$ is either a configuration of (a) or (b) in Figure \ref{D'}, or a copy of $B_6$.
In either cases, we can verify that $\rho(D')\leq \frac{6|D'|-12}{5|D'|}$.
Therefore, $\rho(D')\leq \frac{6|D'|-12}{5|D'|}$ holds.
Since $\Delta_D-5=\Delta_{D'}$ and $|D|-5=|D'|$, it follows that
$$\rho(D)\leq \frac{\Delta_{D'}+5}{|D'|+5}\leq\frac{\frac{6|D'|-12}{5}+5}{|D'|+5}<\frac{6|D|-12}{5|D|}.$$

\noindent{\em Case 3.2.}  Every hole of $D$ is not a $3$-face.

Note that we can add $k-3$ chords to each $k$-face of $D$ so that the final graph becomes a triangulation. Hence
$$3n-6-e(D)\geq\sum_{F\in\mathcal{H}}(|F|-3),$$
where $\mathcal{H}$ is the set  of holes in $D$ and  $|F|$ denote the number of vertices in the facial cycle of $F$. 
We denote by $m(D)=\sum_{F\in\mathcal{H}}(|F|-3)$ the number of all ``missing edges'' in $D$.
To derive a lower bound of $m(D)$, we employ the discharging method. 
Initially, each hole \( F \) in \( D \) is assigned a charge of \( |F|- 3 \), corresponding to its number of missing chords. This charge is then distributed equally among edges in the facial cycle of $F$.
Suppose that $D$ has $t$ TBs isomorphic to $B_1$ and $s$ TBs isomorphic to $B_6$.
Since all solid TBs of $B_6$ are connected only on junction vertices, it follows that $|D|\geq 3s+3$.

For a TB $T$  and each edge $e$ in the boundary of the outer face of $T$, let $F_e$ denote the unique hole in $D$ whose facial cycle contains $e$.
Since each hole in $D$ is not a cycle, $|F_e|\geq 4$ and $e$ receives $\frac{|F_e|-3}{|F_e|}\geq \frac{1}{4}$ charge from $F_e$.
Since a $B_1$-type TB has 3 boundary edges and a $B_6$-type TB has 6 boundary edges, the total charge received by boundary edges of $D$ is at least $\frac{1}{4}(3t+6s)$. It follows that $D$ can accommodate at least $3t/4+3s/2$ additional chords. Therefore, 
$$3|D|-6-e(D)\geq 3t/4+3s/2.$$
While $e(D)=3t+9s$, so we obtain 
$$|D|\geq \frac{5t}{4}+\frac{7s}{2}+2.$$
Thus,
$$\rho(D)=\frac{\Delta_D}{|D|}\leq \frac{4s+t}{\frac{7s}{2}+\frac{5t}{4}+2}=\frac{16s+4t}{14s+5t+8}.$$
When $s\geq 2$, we observe that 
$$\frac{16s+4t}{14s+5t+8}\leq \frac{8s}{7s+4}.$$ 
Since $3s+3\le |D|$ and $\frac{8s}{7s+4}$ is monotonically increasing with respect to $s$. 
We obtain 
$$\frac{8s}{7s+4}\le \frac{\frac{8}{3}\left(|D|-3\right)}{\frac{7}{3}\left(|D|-3\right)+4}< \frac{6|D|-12}{5|D|}.$$ 
When $s\leq 1$, since $|D|\geq 7$, we have that
$$\rho(D)\le \max\left\{\frac{4t}{5t+8},\frac{16+4t}{22+5t}\right\}<\frac{4}{5}\leq \frac{6|D|-12}{5|D|}.$$
\end{proof}

Now we are ready to prove the Theorem \ref{thm1}.

\noindent{\bf Proof of Theorem \ref{thm1}:}
Let $G$ be an $H_4$-free planar graph with $n\geq 72$ vertices that attains the maximum number of edges among all such graphs. Then $G$ must be connected.
If $G$ is a triangulation, then there is a vertex $v$ of degree at least five, implying $G[N[v]]$ contains $H_4$ as a subgraph, a contradiction. Thus, \( G \) admits a plane embedding where the outer boundary \( \Gamma(G) \) is not a 3-face. 
In such an embedding, each 3-face is contained within a TB, which in turn is contained within a TC.  
Let $D_1, D_2,\cdots, D_t$ denote all the TCs of $G$, and let $\rho_i$ represent the triangle-density of $D_i$ for $i\in[t]$.
If \( D_i \) is isomorphic to \( B_1 \), \( B_6 \), or satisfies \( |D_i| \geq 7 \) while \( D_i \notin \{B_{15}^{(6)},B_{15}^{(8)}, B_{15}^{(10)}\} \), then by Lemma \ref{Sec1-densityoftriangles} and Table \ref{tab:1}, we have:  
\[
\rho_i \leq \frac{6|D_i| - 12}{5|D_i|}.
\]
Moreover, the function \( \frac{6|D_i| - 12}{5|D_i|} \) is monotonically increasing in \( |D_i| \). Consequently, we have $\rho _i\leq \frac{6n-12}{5n}$.
For the remaining cases, \( D_i \) is isomorphic to one of the following:
\[
\{B_2, B_3, B_4, B_5,B_7, B_8, B_9, B^{(4)}_{11}, B^{(6)}_{11}, B^{(5)}_{12},B^{(6)}_{14},B^{(6)}_{15}, B^{(8)}_{15}, B^{(10)}_{15}\}
\]
(it is worth noting that $B^{(6)}_{14}=B_8$ and $B_{15}^{(6)}=B_9$).
In these cases, a straightforward verification shows that $\rho_i\leq \frac{6n-12}{5n}$ holds for $n\geq 72$.
Since $D_i$s are pairwise vertex-disjoint, $\sum_{i\in[t]}{|D_i|}\leq n$. Hence,  
\begin{align*}
f_3(G)=\sum_{i\in[t]}{|D_i|\cdot d_i}
\leq \frac{6n-12}{5n}\sum_{i\in[t]}{|D_i|}
\leq \frac{6n-12}{5}.
\end{align*}
Furthermore, we have
$$2e(G)=\sum if_i(G)\geq 3f_3(G)+4(f(G)-f_3(G))=4f(G)-f_3(G).$$
Combining with Euler's formula $e(G)-f(G)+2=n$, we obtain $$e(G)\leq 2n-4+\frac{1}{2}f_3(G).$$
Substituting the bound for \( f_3(G) \) yields: $$e(G)\leq 2n-4+\frac{1}{2}f_3(G)=\frac{13n}{5}-\frac{26}{5}.$$

To demonstrate the tightness of our inequality, consider the graph family \( \{G_k: k\geq 14 \}\) in Figure \ref{graph_Gk}. It is clear that \( G_k \) is a \( H_4 \)-free plane graph with $|G_k|=4k+2$ and $e(G_k)=\frac{13n}{5}-\frac{26}{5}$. This establishes the sharpness of our upper bound and completes the proof of Theorem 1.\qed
\begin{figure}[ht]
\centering  
\begin{tikzpicture}
  % 定义顶点的坐标
  \coordinate (A) at (0, 0);
  \coordinate (B) at (1, 0);
  \coordinate (C) at (2, 0);
  \coordinate (D) at (3, 0);
  \coordinate (E) at (4, 0);
  \coordinate (F) at (5, 0);
  \coordinate (G) at (6, 0);
  \coordinate (H) at (7, 0);
  \coordinate (A') at (0, 1);
  \coordinate (B') at (1, 1);
  \coordinate (C') at (2, 1);
  \coordinate (D') at (3, 1);
  \coordinate (E') at (4, 1);
  \coordinate (F') at (5, 1);
  \coordinate (G') at (6, 1);
  \coordinate (H') at (7, 1);
  \coordinate (X) at (0.5, 0.5);
  \coordinate (Y) at (2.5, 0.5);
  \coordinate (Z) at (4.5, 0.5);
  \coordinate (W) at (6.5, 0.5);
  \coordinate (P) at (3.5, 2.5);
  \coordinate (Q) at (3.5, -1.5);

  % 绘制边
  \draw (A) -- (B) -- (B') -- (A') -- (A);
  \draw (C) -- (D) -- (D') -- (C') -- (C);
  \draw (E) -- (F) -- (F') -- (E') -- (E);
  \draw (G) -- (H) -- (H') -- (G') -- (G); 
  \draw (B') -- (C);  
  \draw (F') -- (G);
  \draw (-0.4,1)arc (180:0:3.7cm and 1.7cm);
  \draw (-0.4,1)arc (180:270:0.4cm and 1cm);
  % \draw[-] (H') .. controls (-1,5) .. (A);
 \fill (3.3,0.5) circle (0.8pt); % 模拟省略号的一个小圆点
 \fill (3.5,0.5) circle (0.8pt); % 再加一个，表示更多的省
 \fill (3.7,0.5) circle (0.8pt);
 \foreach \point in {(0, 0),(1, 0), (0, 1), (1, 1)} {
  \draw (X) -- \point;
}
 \foreach \point in {(2, 0),(3, 0), (2, 1), (3, 1)} {
  \draw (Y) -- \point;
}
 \foreach \point in {(4, 0),(5, 0), (4, 1), (5, 1)} {
  \draw (Z) -- \point;
} 
\foreach \point in {(6, 0),(7, 0), (6, 1), (7, 1)} {
  \draw (W) -- \point;
}
 \foreach \point in {(0, 0),(1, 0),(2, 0),(3, 0), (4, 0),(5, 0),(6, 0),(7, 0)} {
  \draw (Q) -- \point;
}
 \foreach \point in {(0, 1),(1, 1),(2, 1),(3, 1), (4, 1),(5, 1),(6, 1),(7, 1)} {
  \draw (P) -- \point;}
  % 可选：标记顶点
  \foreach \point in {A, B, C, D, E, F, G, H, A', B', C', D', E', F', G', H', X, Y, Z, W, P, Q}
    \fill (\point) circle (1.5pt);
    \node[below] at (3.5,0.3) {$k$};

  % 可选：添加标签
  % \node[below] at (A) {$A$};
  % \node[below] at (B) {$B$};
  % \node[above] at (C) {$C$};
  % \node[left] at (D) {$D$};
  % \node[below left] at (E) {$E$};
  % \node[left] at (F) {$F$};
  % \node[right] at (G) {$G$};
\end{tikzpicture}
 \caption{The extremal graph $G_k$.}
 \label{graph_Gk}
\end{figure}
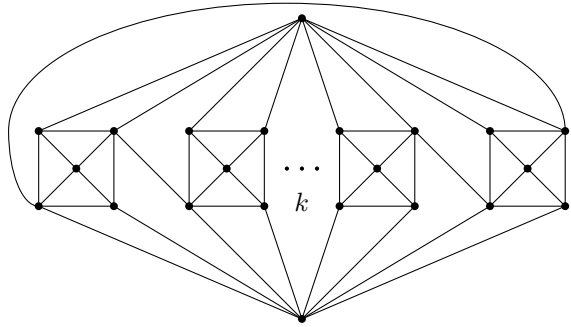

%%%%%%%%%%%%%%%%%%%%%%%%%%%%%%%%%%%%%%%%%%%%%%%%%%%%%%%%%%%%%%%%%%%%%%%%%%%%%%%%%%%%%%%%%
\section{Proof of Theorem \ref{thm2}}\label{sec:4}

This section is dedicated to studying the planar Tur\'{a}n number of $H_5$. We begin by introducing a structural property of $H_5$-free solid TBs, then fully characterize these solid TBs, and finally determine the triangle-densities of $H_5$-free TCs. 
%In the following proof, a $k$-vertex is a vertex of degree $k$, and a $k^+$-vertex has degree at least $k$.
\begin{pro}\label{sec-2-pro-0}
Let \(B\) denote an \(H_5\)-free solid TB with \(|B|\geq4\), and let \(C\) represent its outer boundary. Then, within the vertex set \(V(C)\) or among the holes of \(B\), there exists a vertex \(v\) such that \(B - v\) is a solid TB of order \(|B|-1\).
\end{pro}
\begin{proof}
As discussed in Proposition \ref{sec-1-pro-0}, the result holds when $B$ does not contain holes.
Hence, assume that $B$ contains a hole.
Let $\mathcal{F}$ be the set of holes of $B$.
Then for each $F\in \mathcal{F}$, $|V(\partial(F))\cap V(C)|\leq 1$.
Since $B$ is a solid TB,  each $\partial(F)$ is a cycle of order at least four.

\noindent\textbf{Case 1.} For each $F\in \mathcal{F}$, $|V(\partial(F))\cap V(C)|=0$. 

We employ Fact 1 in Proposition \ref{sec-1-pro-0}, which still hold in this case. If there exists a vertex \( v \in V(C) \) with degree two, then \( B - v \) remains a solid TB. Thus, we assume that each vertex in $C$ has degree at least three. If there is a vertex $v\in V(C)$ such that $d(v)=t\geq 4$, then let $x_1,x_2,\ldots,x_t$ be neighbors of $v$ listed in counter-clockwise order around $v$ such that $x_1,x_t\in V(C)$. 
Note that either $x_1x_{t-1}\notin E(B)$ or $x_tx_2\notin E(B)$.
Hence, either $\Theta_{x_1x_2}\cup vx_{t-1}x_tv$ or  $\Theta_{x_{t-1}x_t}\cup vx_1x_2v$ is an $H_5$, a contradiction.
Thus, each vertex in $C$ is a  $3$-degree vertex. Let \( C = v_1v_2v_3 \cdots v_{|C|}v_1 \), and let \( u \) denote the third neighbor of \( v_1 \) (distinct from \( v_2 \) and \( v_{|C|} \)). Since \( v_{|C|}uv_2 \) is a path,  \( u \) is adjacent to \( v_2 \). Proceeding inductively, for each $i\in[|C|]$, `we have that \( u \) is adjacent to \( v_i \). Consequently, \( B \) is  a wheel graph \( W_{|C|} \), contradicting $B$ contains a hole.

\noindent\textbf{Case 2.} there exists a face $F\in \mathcal{F}$ such that $|V(\partial(F))\cap V(C)|=1$.

Let $v\in V(C)$ be a vertex incident to $a$ non--3-faces, denoted by $F_1, \ldots, F_a$ in counter-clockwise order around $v$. Since each edge of $B$ lies in at least one 3-face, there is a 3-face between any two adjacent faces $F_i$ and $F_{i-1}$ for $i\in[a]$ (for convenience, we define $F_0$ as the outer face of $B$). We aim to show that either $B-v$ is a solid TB or $B-u$ is a solid TB for some $u\in \partial(F)$, where $F$ is a hole of $B$.

First, we claim that $B[N[v]]$ is a friendship graph. 
Otherwise, assume there exists an $i\in[a]$ such that there are $r\geq 2$ $3$-faces between $F_{i-1}$ and $F_i$.
Say $x_1x_2\ldots x_{r+1}$ are consecutive neighbors of $v$ listed in counter-clockwise order around $v$ such that $x_1\in V(\partial(F_{i-1}))$ and $x_{r+1}\in V(\partial(F_i))$, where $i$ take mudola $a$.
Similarly, assume that $x'_1x'_2\ldots x'_{\ell+1}$ be consecutive neighbors of $v$ listed in counter clockwise order around $v$ such that $x'_1\in V(\partial(F_{i-2}))$ and $x'_{\ell+1}\in V(\partial(F'_{i-1}))$.
Then $vx_ix_{i+1}v$ and $vx'_jx'_{j+1}v$ are $3$-faces of $B$ for $i\in[r]$ and $j\in[\ell]$.
Let $y_1$ be the neighbor of $x_1$ on $F_{i-1}$ other than $v$.
Note that $x_1y_1$ is in the unique $3$-face of $B$, say $x_1y_1z_1x_1$.
We will complete the proof by several cases as follows.
\begin{enumerate}
\item [(1)] $z_1=x_2$.\\
    Let $y_2$ be the neighbor of $y_1$ on $\partial(F_{i-1})$ such that $y_2\neq x_1$.
    Note that $y_1\neq x'_{\ell+1}$ since $F_{i-1}$ is not a $3$-face.
    Without loss of generality, assume that $F^*$ is the unique $3$-face of $B$ containing $y_1y_2$, say $\partial(F^*)=y'y_1y_2y'$.
    If $x'_{\ell}=y_1$, then $y_2=x'_{\ell+1}$ and $y'=v$. One can verify that $B-y_2$ is a desired solid TB.  
    % then since the cycle $vx'_{\ell+1}y_1x_1v$ divide the plane into two parts, it follows that $y'\notin V(\Theta_{x_1x_2})$.
    % Hence, $\Theta_{x_1x_2}\cup y'y_1y_2y'$ is a copy of $\Theta_4\cup C_3$, a contradiction.
    If $x'_{\ell}\neq y_1$, then $\Theta_{x_1x_2}\cup vx'_{\ell}x'_{\ell+1}v$ is a copy of $\Theta_4\cup C_3$, a contradiction.
\item [(2)] $z_1=x_3$.\\
    Since $x'_{\ell}\neq x_i$ and $x'_{\ell+1}\neq x_i$ for each $i\in[r+1]$, it follows that $vx_1x_2v\cup vx'_{\ell}x'_{\ell+1}v\cup\{x_2x_3,x_1x_3\}$ is a copy of $\Theta_4\cup C_3$, a contradiction.
\item [(3)] $z_1\notin\{x_2,x_3\}$.\\
    Then $x_1y_1z_1x_1\cup \Theta_{vx_2}$ is a copy of $\Theta_4\cup C_3$, a contradiction.
\end{enumerate}
Therefore,  $B[N[v]]$ is a friendship graph.

Now we show that $B-v$ is a solid TB. 
This proof is similar to the proof in Case 2 of Proposition \ref{sec-1-pro-0}, and is restated here for convenience.
Let $F'$ and $F''$ be two 3-faces in $B-v$, and let $\mathcal{P}$ be the set of connecting sequences between them in $B$. If some alternating sequence $P \in \mathcal{P}$ avoids 3-faces containing $v$, then $F' \sim F''$ in $B-v$. Otherwise, every $P \in \mathcal{P}$ includes a 3-face $F$ containing $v$, and thus contains a sub-alternating sequence $F^*eFeF^*$, where  $F^*$ is a 3-face with $F^*\cap F=\{e\}$. Removing all such sub-alternating sequences from $P$ yields a residual alternating sequence $P'$ connecting $F'$ and $F''$ within $B-v$. Therefore, $F'$ and $F''$ are triangular-connected in $B-v$, implying $B-v$ is a solid TB.
\end{proof}

\begin{lem}\label{sec-2-lem1}
 Let \( B \) be an \( H_5 \)-free solid TB. Then \( B \) is isomorphic to one of the following configurations:  
(i) configurations \( B_1 \)-\( B_6 \), \( B_{11}^{(4)} \), and \( B_{12}^{(5)} \) in Figure \ref{fig:H4free};  
(ii) $k$-wheel \( W_k \) and $k$-fan \( F_k \) (\( k \geq 5 \));  
(iii) the three exceptional configurations $B_1',B_2'$ and $B_3'$ illustrated in Figure \ref{fig:H5free}.
\end{lem}

 \begin{proof}
 All solid TBs of order at most 5 are $H_5$-free, as demonstrated by the structures $B_1$ to $B_5$, $B_{11}^{(4)}$, $B_{12}^{(5)}$ in Figure \ref{fig:H4free}.
If $|B|\geq 6$, Proposition \ref{sec-2-pro-0} ensures there exists a vertex $v\in \partial(F)$ such that $B-v$ is a solid TB, where $F$ is either the outer face of $B$ or a hole of $B$. 

\begin{figure}[ht] 
     \centering  
     \subfloat[{$B'_1$}]
    {\begin{tikzpicture}
    \pgfmathparse{sqrt(3)}
    [inner sep=0.8pt]	
                 \node[circle, fill, inner sep=1.5pt](v0) at (0,0)[]{};
               \node[circle, fill, inner sep=1.5pt](v1) at (\pgfmathresult,0)[]{};
                \node[circle, fill, inner sep=1.5pt](v2) at (\pgfmathresult/2,3/2)[]{};
                \node[circle, fill, inner sep=1.5pt](v3) at (\pgfmathresult/2,1/2)[]{};
                \node[circle, fill, inner sep=1.5pt](v4) at (0,1)[]{};
                \node[circle, fill, inner sep=1.5pt](v5) at (\pgfmathresult,1)[]{};
                \draw[-] (v0) -- (v1);
                \draw[-] (v1) -- (v2);
                \draw[-] (v2) -- (v0);
                \draw[-] (v3) -- (v1);
                \draw[-] (v3) -- (v2);
                \draw[-] (v3) -- (v0);
                \draw[-] (v0) -- (v4) -- (v2);
                \draw[-] (v1) -- (v5) -- (v2);
            \end{tikzpicture}
            } \hspace{3em}\subfloat[{$B'_2$}]
   {\begin{tikzpicture}
    \pgfmathparse{sqrt(3)}
    [inner sep=0.1mm]	
               \node[circle, fill, inner sep=1.5pt](v0) at (0,0)[]{};
               \node[circle, fill, inner sep=1.5pt](v1) at (\pgfmathresult,0)[]{};
                \node[circle, fill, inner sep=1.5pt](v2) at (\pgfmathresult/2,3/2)[]{};
                \node[circle, fill, inner sep=1.5pt](v3) at (\pgfmathresult/2,0.9)[]{};
                \node[circle, fill, inner sep=1.5pt](v4) at (\pgfmathresult/2,0.38)[]{};
                \node[circle, fill, inner sep=1.5pt](v5) at (\pgfmathresult/2,-0.38)[]{};
                 
                \draw[-] (v0) -- (v1) -- (v2) -- (v0);
                \draw[-] (v2) -- (v3) -- (v4);
                \draw[-] (v0) -- (v3) -- (v1);
                \draw[-] (v0) -- (v4) -- (v1);
                \draw[-] (v0) -- (v5) -- (v1);
            \end{tikzpicture}
    } \hspace{3em}
     \subfloat[{$B'_3$}]
    {\begin{tikzpicture}
    \pgfmathparse{sqrt(3)}
    [inner sep=0.8pt]	
                 \node[circle, fill, inner sep=1.5pt](v0) at (0,0)[]{};
               \node[circle, fill, inner sep=1.5pt](v1) at (\pgfmathresult,0)[]{};
                \node[circle, fill, inner sep=1.5pt](v2) at (\pgfmathresult/2,3/2)[]{};
                \node[circle, fill, inner sep=1.5pt](v3) at (\pgfmathresult/2,1/2)[]{};
                \node[circle, fill, inner sep=1.5pt](v4) at (0,1)[]{};
                \node[circle, fill, inner sep=1.5pt](v5) at (\pgfmathresult,1)[]{};
                \node[circle, fill, inner sep=1.5pt](v6) at (\pgfmathresult/2,-1/2)[]{};
                \draw[-] (v0) -- (v1);
                \draw[-] (v1) -- (v2);
                \draw[-] (v2) -- (v0);
                \draw[-] (v3) -- (v1);
                \draw[-] (v3) -- (v2);
                \draw[-] (v3) -- (v0);
                \draw[-] (v0) -- (v4) -- (v2);
                \draw[-] (v1) -- (v5) -- (v2);
                \draw[-] (v0) -- (v6) -- (v1);
            \end{tikzpicture}
             }    \caption{\label{fig:H5free} A part of $H_5$-free triangle-blocks.}
    \end{figure}
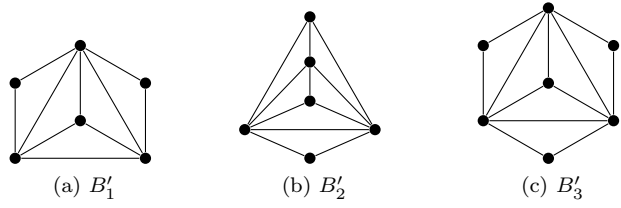

For $|B|=6$, $B-v \in \{B_3,B_4,B_5,B_{12}^{(5)}\}$. 
If $B-v$ is a copy of $B_3$, then $B\cong B'_1$;
if $B-v$ is a copy of $B_4$, then $B$ contains $H_5$ as a subgraph, a contradiction;
if $B-v$ is a copy of $B_5$, then $B\cong B'_2$;
if $B-v$ is a copy of $B_{12}^{(5)}$, then $B\in  \{F_{5},W_{5},B_6\}$.

For $|B|=7$, $B-v$ is a solid TB belonging to $\{B_6,B'_1,B'_2,W_5,F_5\}$.
If $B-v$ is a copy in $\{B_6, B'_2,W_5\}$, then $B$ contains $H_5$ as a subgraph, a contradiction;
If $B-v$ is a copy of $B'_1$, then $B\cong B'_3$;
If $B-v$ is a copy of $F_5$, then $B\in \{F_6,W_6\}$.

For $|B|=8$, $B-v$ is a solid TB in $\{B'_3,W_6,F_6\}$.
If $B-v$ is a copy of $B'_{3}$ or $W_6$, then $B$ must contain $H_5$ as a subgraph, a contradiction.
If $B-v$ is a copy of  $F_6$, then $B$ is isomorphic to one solid TB in $\{F_7,W_7\}$.

For $|B|=9$, $B-v$ is a solid TB with $B-v\in \{W_7,F_7\}$.
If $B-v$ is a copy of $W_7$, then $B$ must contain $H_5$ as a subgraph, a contradiction.
If $B-v$ is a copy of $F_7$, then $B$ is isomorphic to one solid TB in $\{F_8,W_8\}$;

Through inductively construction, we establish that for $|B|=k+1$ with $k\geq 9$, $B\in \{F_k,W_k\}$.
This completes this proof.
\end{proof}

For convenience, let $\mathcal{F}$ denote the set of all $H_5$-free solid TBs, then 
\begin{center}
    $\mathcal{F}=\{B_1,B_2,B_3,B_4,B_5,B_6,B^{(4)}_{11},B^{(5)}_{12},B'_1,B'_2,B'_3\}\cup \{W_{k},F_k:k\geq 5\}$.
\end{center}

\begin{lem}\label{Sec2-densityoftriangles}
If $D$ is an $H_5$-free TC consisting of solid TBs, then the triangle-density of $D$ is at most $1$.
Moreover, the triangle-density of $D$ is one only if $D$ is a copy of $B_5$ or $B_2'$. 
\end{lem}

\begin{proof}
Table \ref{table:H5free} shows triangle-densities of $B_1',B_2',B_3',W_n$ and $F_n$.
Combining Table \ref{tab:1}, we have that the triangle-density of $D$ is at most $1$ when $D$ is a solid TB.
Moreover, the triangle-density of $D$ is one only if $D\in\{B_5,B_2'\}$.
Next, we assume that $D$ consists of at least two at least TBs.
Our aim is to show that $\rho(D)<1$.
For two solid TBs $B',B''$ of $D$ with $|V(B')\cap V(B'')|\geq 1$, there is a $3$-face $F$ of $B''$ such that $|V(B')\cap V(\partial(F))|\geq 1$.
Choose such solid TBs $B',B''$ and a $3$-face $F$ of $B''$ such that $|V(B')\cap V(\partial(F))|$ is maximum.

\noindent{\bf Case 1.} $|V(B')\cap V(\partial(F))|=1$. 

By the choice of $B',B''$ and the face $F$ of $B''$, we have that $B',B''$ are fan graphs, and $V(B')\cap V(B'')=\{v\}$ is the common center vertex of them.
If there is a TB $B=F_k$ for some $k\geq 3$, then the center vertex of $B$, say $u$, is a cut-vertex of $D$.
Let $D'=D-(V(B)-v)$.
By induction, $\rho(D')\leq 1$.
Since $\Delta_B=|B|-2$ and $|D|=|D'|+|B|-1$, it follows that $\rho(D)<1$.
Thus, we assume that each TB of $D$ is an $F_2=B_1$.
Let $\mathcal{F}$ denote the set of all faces of $D$ not in any TBs.
For each $F\in \mathcal{F}$, let $\ell_F$ denote the length of $\partial(F) $ (note that $\partial(F) $ is a closed trail).
If there is a face $F$ such that $F$ is a $3$-face, then since  $D$ is $H_5$-free,  $D$ consists of three $B_1$s. Consequently,  $\rho(D)=1/2<1$.
Now, we assume that each $F\in \mathcal{F}$ is not a $3$-face.
Note that 
\begin{align}\label{ineq-bu}
e(D)+\sum_{F\in \mathcal{F}}(\ell_F-3)=3n-6.
\end{align}
Assign a charge of  $\ell_F-3$ to each $F\in \mathcal{F}$, and then distribute this charge equally among vertices in $\partial(F)$ (since $\partial(F)$ is a closed trail, the charge on some vertices may be counted more than once).
Since each $F\in \mathcal{F}$ is not a $3$-face,  if a vertex $v\in \partial(F)$ appears $k$ times in  $\partial(F)$, then $v$ receives $k(\ell_F-3)/\ell_F\geq k/4$ charge.
For each vertex of $v\in V(D)$, assume that there are $d_v$ TBs incident with $v$.
Then $v$ receives $d_v(\ell_F-3)/\ell_F\geq d_v/4$ charge.
Therefore, 
$$\sum_{v\in V(D)}d_v=e(D)$$
and
$$ \sum_{F\in \mathcal{F}}(\ell_F-3)=\sum_{v\in V(D)}d_v(\ell_F-3)/\ell_F\leq \frac{1}{4}\sum_{v\in V(D)}d_v=\frac{1}{4}e(D).$$
Combining with Ineq. (\ref{ineq-bu}), we have that $e(D)\leq (12|D|-24)/5$.
Since each TB in $D$ is a $B_1$-type TB and any two TBs are edge-disjoint, it follows that there are at most $e(D)/3<|D|$ TBs in $D$.
Hence, $\rho(D)=\Delta(D)/|D|<1$.

\noindent{\bf Case 2.} $|V(B')\cap V(\partial(F))|\geq 2$.

 Then $B'\in\{B_{11}^{(4)},B_4,F_4,F_5,B_6\}$ and $B'\cup \partial(F)$ is a graph as shown in Figure \ref{pict-thm-2} (a)--(e). respectively;  otherwise $B'\cup B''$ contains an $H_5$, a contradiction.
Since $D$ is $H_5$-free, we have that $B''$ is either a $B_1$-type TB (see Figure \ref{pict-thm-2} (a)--(e)) or a $B_{11}^{(4)}$-type TB (see Figure \ref{pict-thm-2} (A)--(E)), 
and $D=B'\cup B''$. Consequently,  the triangle-density of $D$ is less that $1$.
\end{proof}

\begin{figure}[ht]
	\centering
	\includegraphics[width=350pt]{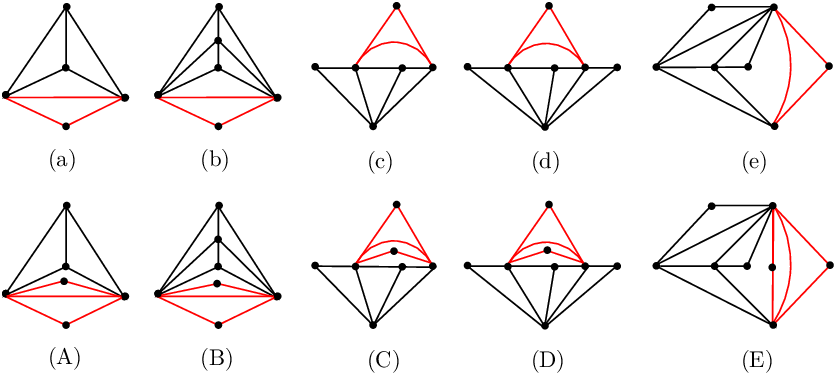}\\
	\caption{All possible TBs of $B'$.} \label{pict-thm-2}
\end{figure}

\begin{table}[htbp]
    \centering
    \begin{tabular} 
    {|c|c|c|c|c|c|c|c|c|c|c|c|c|} % 表格的列由竖线分隔，c表示居中对齐
  \hline % 第一行水平线（表头下面的线）
  Cases & $B'_1$ & $B'_2$ & $B'_3$ & $W_k$ & $F_{k+1}$ \\ % 第一行内容
   \hline % 第二行水平线（第一行内容下面的线）
   $\Delta_{B}$($B\in \{B'_i(i\in[3]),W_k,F_{k+1}$) $\leq $& 5 & 6  & 6 & $k$ & $k$  \\ % 第二行内容
  \hline % 第三行水平线（第二行内容下面的线）
  Triangle density $\leq$ & $\frac{5}{6}$ & 1 & $\frac{6}{7}$ & $\frac{k}{k+1}$ & $\frac{k}{k+2}$ \\ % 第三行内容
  \hline % 最后一行水平线（表格底部的线，可选）
\end{tabular}
\caption{\label{table:H5free}The triangle-densities of a part of $H_5$-free triangle-components.}
\end{table} 

Now we are ready to give the proof of Theorem \ref{thm2}.

\noindent{\bf Proof of Theorem \ref{thm2}}
Let \( G \) be an \( H_5\)-free plane graph with \( n \geq 6 \) vertices and the maximum number of edges, which ensures \( G \) is connected. 
If $G$ is a triangulation, then $G$ itself is an $H_5$-free TB, implying $G$ is a copy of $B_1,B_2$ or $B_5$.
This contradicts $n\geq 6$.
Hence, we can assume that \( G \) is embedded in the plane such that its outer face  is not a 3-face. Let \( D_1, D_2, \ldots, D_t \) denote all TCs in \( G \), and let \( \rho_i \) represent the triangle-density of \( D_i \). Since the outer boundary of $G$ is not a $3$-face, each $3$-face is an inner face of some TB in $G$. Hence, 
\begin{align}\label{thm2-eq-1}
f_3(G)=\sum_{i\in[t]}\Delta_{D_i}=\sum_{i\in[t]}|D_i|\rho(D_i)\leq \sum_{i\in[t]}|D_i|\leq n,
\end{align} 
Then, 
\begin{align}\label{thm2-eq-2}
2e(G)=\sum_{i\geq 3} if_i(G)\geq 3f_3(G)+4(f(G)-f_3(G))=4f(G)-f_3(G).
\end{align}
Combining with Euler's formula $e(G)-f(G)+2=n$, we obtain $$e(G)\leq 2n-4+\frac{1}{2}f_3(G).$$
Thus, 
\begin{align}\label{thm2-eq-3}
e(G)\leq 2n-4+\frac{1}{2}f_3(G)=\frac{5n}{2}-4.
\end{align}

To demonstrate the sharpness of the inequality, let $k\geq 4$ be an even integer and \( R'_k \) be the plane graph shown in Figure \ref{fig:M}, constructed from \( k \) disjoint \( B_5 \) copies augmented with two cycles \( C_k := v_1v_2\cdots v_kv_1 \) and \( C_{2k} := u_1u_2\cdots u_{2k}u_1 \). Let \( R_k \) be derived from \( R'_k \) by adding edge $\{v_iv_{k+1-i}:i\in[k/2]\}\cup \{u_iu_{2k+1-i}:i\in[k]\}$. 
It is clear that $R_k$ is an $H_5$-free plane graph.

Now we construct an extremal graph when $n=10x+6y$ has integer solutions  $x\geq 2$ and $y\geq 0$. 
Let $H_0=R_{x}$ and let $H_i$ be a plane graph obtained from $H_{i-1}$ by adding a copy of $B_2'$, say $B$, in a $4$-face $F$ of $H_{i-1}$, and then joining each vertex of $V(\partial(F))$ to a vertex on the outer boundary of $B$ such that the four new edges forms a matching. 
Then $H_y$ is an $n$-vertex $H_5$-free plane graph. \qed

\smallskip

\noindent{\bf Remark 1.}
Note that $e(G)=\frac{5n}{2}-4$ holds if and only if all equalities in (\ref{thm2-eq-1}), (\ref{thm2-eq-2}) and (\ref{thm2-eq-3}) hold, which implies 
\begin{enumerate}
    \item [(i)] each $D_i$ is a copy of $B_5$ or $B_2'$, 
    \item [(ii)] $\bigcup_{i\in[t]}V(D_i)=V(G)$, and 
    \item [(iii)] each face of $G$ is either a $3$-face or a $4$-face.
\end{enumerate}
Therefore, if $n=10x+6y$ has integer solutions  $x\geq 2$ and $y\geq 0$, then each graph satisfies conditions   (i), (ii) and (iii).
The graph $H_y$ constructed above implies that such graphs exist definitely. \qed

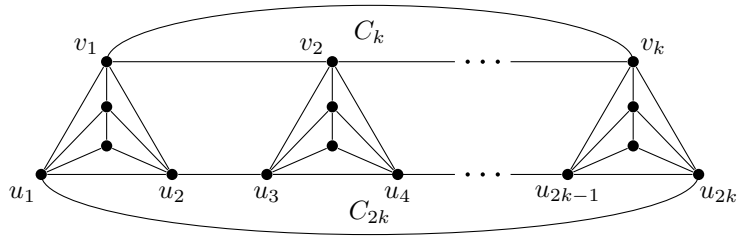
\begin{figure}[ht] 
     \centering  
{\begin{tikzpicture}
\pgfmathparse{sqrt(3)}
[inner sep=0.1mm]	
\node[circle, fill, inner sep=1.5pt](v0) at (0,0)[]{};
\node[circle, fill, inner sep=1.5pt](v1) at (1.74,0)[]{};
\node[circle, fill, inner sep=1.5pt](v2) at (0.87,3/2)[]{};
\node[circle, fill, inner sep=1.5pt](v3) at (0.87,0.9)[]{};
\node[circle, fill, inner sep=1.5pt](v4) at (0.87,0.38)[]{};
\draw[-] (v0) -- (v1) -- (v2) -- (v0);
\draw[-] (v2) -- (v3) -- (v4);
\draw[-] (v0) -- (v3) -- (v1);
\draw[-] (v0) -- (v4) -- (v1);
\node at (0.6,1.7) {$v_1$};
\node at (-0.25,-0.27) {$u_1$};
\node at (1.74,-0.27) {$u_2$};
                
\node[circle, fill, inner sep=1.5pt](u0) at (3,0)[]{};
\node[circle, fill, inner sep=1.5pt](u1) at (4.74,0)[]{};
\node[circle, fill, inner sep=1.5pt](u2) at (3.87,3/2)[]{};
\node[circle, fill, inner sep=1.5pt](u3) at (3.87,0.9)[]{};
\node[circle, fill, inner sep=1.5pt](u4) at (3.87,0.38)[]{};
\draw[-] (u0) -- (u1) -- (u2) -- (u0);
\draw[-] (u2) -- (u3) -- (u4);
\draw[-] (u0) -- (u3) -- (u1);
\draw[-] (u0) -- (u4) -- (u1);
\node at (3.6,1.7) {$v_2$};
\node at (3,-0.27) {$u_3$};
\node at (4.74,-0.27) {$u_4$};

\node[circle, fill, inner sep=1.5pt](w0) at (7,0)[]{};
\node[circle, fill, inner sep=1.5pt](w1) at (8.74,0)[]{};
\node[circle, fill, inner sep=1.5pt](w2) at (7.87,3/2)[]{};
\node[circle, fill, inner sep=1.5pt](w3) at (7.87,0.9)[]{};
\node[circle, fill, inner sep=1.5pt](w4) at (7.87,0.38)[]{};
\draw[-] (w0) -- (w1) -- (w2) -- (w0);
\draw[-] (w2) -- (w3) -- (w4);
\draw[-] (w0) -- (w3) -- (w1);
\draw[-] (w0) -- (w4) -- (w1);
                \node at (8.14,1.7) {$v_k$};
                
                \fill (5.67,0) circle (0.8pt); % 模拟省略号的一个小圆点
               \fill (5.87,0) circle (0.8pt); % 再加一个，表示更多的省
               \fill (6.07,0) circle (0.8pt);
               \fill (5.67,1.5) circle (0.8pt); % 模拟省略号的一个小圆点
               \fill (5.87,1.5) circle (0.8pt); % 再加一个，表示更多的省
               \fill (6.07,1.5) circle (0.8pt);
              \draw[-] (v2) -- (u2) -- (5.5,1.5);
              \draw[-] (6.24,1.5) -- (w2);
              \draw[-] (v1) -- (u0) -- (5.5,0);
              \draw[-] (6.24,0) -- (w0);
            \node at (7,-0.27) {$u_{2k-1}$};
            \node at (9.01,-0.27) {$u_{2k}$};
            \draw (0.87,3/2)arc (180:0:3.5 and 0.75);
            \draw (0,0)arc (-180:0:4.37 and 0.8);
            \node at (4.37,1.9) {$C_k$};
            \node at (4.37,-0.5) {$C_{2k}$};
            \end{tikzpicture}\caption{\label{fig:M}The graph $R'_k$}} 
\end{figure}

%%%%%%%%%%%%%%%%%%%%%%%%%%%%%%%%%%%%%%%%%%%%%%%%%%%%%%%%%%%%%%%%%%%%%%%%%%%%%%%%%%%%%%%%%

\section{Proof of Theorem \ref{thm: C3Theta4}}\label{sec:5}

In this section, we study the planar Tur\'an number of $C_3\dot{\cup} \Theta_4$. 
Let $G$ be a $C_3\dot{\cup}  \Theta_4$-free plane graph with $|G|\geq174$. 
A set of edges is called {\em independent edges} if they form a matching. 
The proof of Theorem \ref{thm: C3Theta4} could be proceeded using the idea of analyzing the size of $E_I(G)$ proposed in \cite{LI2024260}. 
The following result is necessary to avoid $C_3\dot{\cup}  \Theta_4$ in $G$.
\begin{lem}\label{sec-3-lem1}
For any two independent edges $e$, $f$ in $E_I(G)$, $|V(\Theta_e)\cap V(\Theta_f)|\geq2$. 
\end{lem}
Let $e=uv$ be an edge in $E_I(G)$ and $V(\Theta_e)=\{u,v,x,y\}$. 
Define $A_e^*=\{\text{$e\in E_I(G)$: $e$}$ is incident to at least one vertex in $V(\Theta_e)\}$ and let $A_e=E_I(G)\setminus A_e^*$. We provide the following observation.

\begin{obs}\label{sec-3-lem2}
If $f\in A_e$, then $\Theta_e\cup \Theta_f$ must be isomorphic to one of the structures $\{D_i:i\in [3]\}$  illustrated in Figure \ref{fig:Q}. 
\end{obs}

\begin{figure}[ht] 
     \centering  
     \subfloat[{$D_1$}]
   {\begin{tikzpicture}
    \pgfmathparse{sqrt(3)}
    [inner sep=0.1mm]	
                \node[circle, fill, inner sep=1.5pt](u) at (-1,0.6)[]{};
               \node[circle, fill, inner sep=1.5pt](v) at (1,0.6)[]{};
                \node[circle, fill, inner sep=1.5pt](a) at (0,-0.5)[]{};
                \node[circle, fill, inner sep=1.5pt](b) at (0,-1.3)[]{};
                \node[circle, fill, inner sep=1.5pt](x) at (0,1.1)[]{};
                \node[circle, fill, inner sep=1.5pt](y) at (0,0.1)[]{};
                \draw[-] (u) -- (v);
                \draw[-] (a) -- (b);
                \foreach \point in {(x),(y),(a),(b)} {
               \draw (u) -- \point;}
                \foreach \point in {(x),(y),(a),(b)} {
               \draw (v) -- \point;}
                 \node at (0,0.78) {$e$};
               \node at (0.15,-0.75) {$f$};
            \node at (-1.25,0.6) {$u$};
            \node at (1.25,0.6) {$v$};
            \node at (0,-0.25) {$a$};
            \node at (0,-1.55) {$b$};
            \node at (0,1.35) {$x$};
            \node at (0,0.35) {$y$};
            \end{tikzpicture}
    } \hspace{4em}
     \subfloat[{$D_2$}]
    {\begin{tikzpicture}
    [inner sep=0.8pt]	
                \node[circle, fill, inner sep=1.5pt](u) at (-0.1,0)[]{};
               \node[circle, fill, inner sep=1.5pt](v) at (0.6,0)[]{};
                \node[circle, fill, inner sep=1.5pt](a) at (1.5,0)[]{};
                \node[circle, fill, inner sep=1.5pt](b) at (2.2,0)[]{};
                \node[circle, fill, inner sep=1.5pt](x) at (1.05,1.2)[]{};
                \node[circle, fill, inner sep=1.5pt](y) at (1.05,-1.2)[]{};
                \draw[-] (u) -- (v);
                \draw[-] (a) -- (b);
                \foreach \point in {(u),(v),(a),(b)} {
               \draw (x) -- \point;}
                \foreach \point in {(u),(v),(a),(b)} {
               \draw (y) -- \point;}
            \node at (0.35,0.18) {$e$};
            \node at (1.75,0.18) {$f$};
            \node at (-0.35,0) {$u$};
            \node at (0.85,0) {$v$};
            \node at (1.25,0) {$a$};
            \node at (2.45,0) {$b$};
            \node at (1.05,1.45) {$x$};
            \node at (1.05,-1.45) {$y$};
            \end{tikzpicture}
            } \hspace{4em}
     \subfloat[{$D_3$}]
    {\begin{tikzpicture}
    [inner sep=0.8pt]	
                  \node[circle, fill, inner sep=1.5pt](x) at (-0.8,0.6)[]{};
               \node[circle, fill, inner sep=1.5pt](v) at (0.8,0.6)[]{};
                \node[circle, fill, inner sep=1.5pt](a) at (0,-0.3)[]{};
                \node[circle, fill, inner sep=1.5pt](b) at (0,-1.3)[]{};
                \node[circle, fill, inner sep=1.5pt](u) at (-0.8,1.1)[]{};
                \node[circle, fill, inner sep=1.5pt](y) at (0.8,1.1)[]{};
                \draw[-] (y) -- (v);
                \draw[-] (y) -- (u);
                \draw[-] (a) -- (b);
                \foreach \point in {(u),(v),(a),(b)} {
               \draw (x) -- \point;}
                \foreach \point in {(u),(a),(b)} {
               \draw (v) -- \point;}
                 \node at (0.3,0.9) {$e$};
               \node at (0.15,-0.6) {$f$};
            \node at (1.05,0.6) {$u$};
            \node at (-1.05,1.1) {$v$};
            \node at (0,-0.05) {$a$};
            \node at (0,-1.55) {$b$};
            \node at (-1.05,0.6) {$x$};
            \node at (1.05,1.1) {$y$};
            \end{tikzpicture}
             }    \caption{\label{fig:Q} The planar structures constituted by $\Theta_e \cup \Theta_f$ are $D_1$, $D_2$, and $D_3$. Specifically, $D_1$ contains two holes, namely $auyva$ and $xubvx$. $D_2$ contains two holes, which are $xvyax$ and $xuybx$. And $D_3$ contains two holes, $axua$ and $buyvxb$.}
    \end{figure}
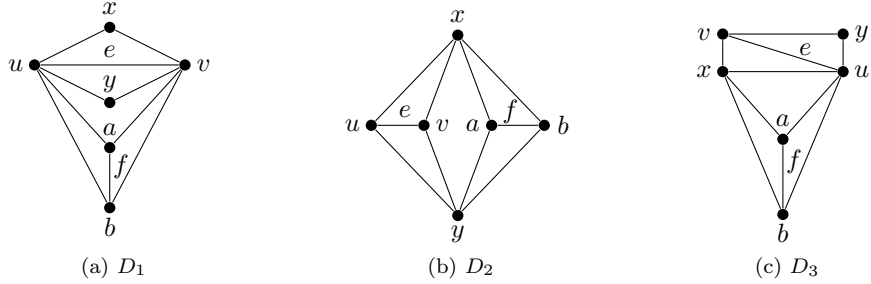

Let $\{e_1,\cdots,e_t\}\subseteq E_I(G)$. For any inner face $F$ of the plane subgraph $H =\cup_{i\in[t]} \Theta_{e_i}$, if $F$ is
not a 3-face in any $\Theta_{e_i}$, then $F$ is call a {\em pseudo face} of $H$.
An edge or a vertex is said to \textit{lie in} a pseudo face when it is contained in the interior region of the face's closed boundary in the plane.
Given a plane subgraph $P$ of $G$ and an edge subset $S=\{e_1,\ldots,e_t\}\subseteq E_I(G)$, the \emph{generating graph} of $P$ by $S$ is $P' = P \cup \bigcup_{i=1}^t \Theta_{e_i}$. If $S$ consists of a single edge $e$, we say $P' = P \cup \Theta_e$ is the generating graph of $P$ by $e$.
Let $d_{I}(v)=d_{G[E_I(G)]}(v)$ and $\Delta_I(G)=\Delta(G[E_I(G)])$. 
We now establish a key lemma that will be essential for proving Theorem \ref{thm: C3Theta4}. 
% Recall that $\frac{n-2}{2}K_2$ represents an $n$-vertex graph whose edge set consists of $\left\lfloor\frac{n-2}{2}\right\rfloor$ independent edges.
\begin{lem}\label{lem:theupperboundof|E_I|}
Let $G$ be a $C_3\dot{\cup}  \Theta_4$-free planar graph with order $n\geq 174$. If $|E_I(G)|>\frac{n}{2}$ and $\Delta_I(G)\leq 9$, then $|E_I(G)|\leq \frac{n}{2}+4$.
Moreover, if $|E_I(G)|=\left\lfloor\frac{n}{2}\right\rfloor+4$, then  $G$ contains  a subgraph that is isomorphic to $\left(\left\lfloor\frac{n-2}{2}\right\rfloor K_2\right )+K_2$.
\end{lem}

\begin{proof}
If $|E_I(G)|\leq 90$, then since $n\geq 174$, we derive $|E_I(G)|< \left\lfloor\frac{n}{2}\right\rfloor+4$, and the result follows immediately.
We therefore proceed under the assumption that $|E_I(G)|> 90$.	
Since $\Delta_I(G) \leq 9$, for each $e\in E_I(G)$ and each $f\in A_e$, $A_e$ must contain at least 4 independent edges including $f$.  Based on the structure of $\Theta_e \cup \Theta_f$ shown in Observation \ref{sec-3-lem2}, our analysis proceeds by examining three distinct cases.

\noindent\textbf{Case 1.} There exists an $e\in E_I(G)$ and an $f\in A_e$ such that $\Theta_e\cup \Theta_f$ is a copy of $D_1$.

Without loss of generality, let $f=ab$, $e=uv$ and $V(\Theta_e)=\{u,v,x,y\}$.
Then $V(\Theta_f)=\{u,v,a,b\}$.
Then the two pseudo faces of $\Theta_e\cup \Theta_f$ are $F_1=buxvb$ and $F_2=auyva$.
Since $x,y$ belong to $\partial(F_1)$ and $\partial(F_2)$, respectively, it follows that $xy\notin E(G)$.
Define $B_e^*$ as the set of edges in $E_I(G)$ that are incident with either $u$ or $v$, and let $B_e=E_I(G)-B^*_e$.
Then $f\in B_e$.

\begin{claim}\label{thm-3-clm-1}
	If $f'=a'b'$ is an edge of $B_e$ such that $\{x,y\}\cap\{a',b'\}\neq \emptyset$, then $\{x,y\}\neq\{a',b'\}$ and $V(\Theta_{f'})=\{a',b',u,v\}$.
\end{claim}
\begin{proof}
It is clear that $\{x,y\}\neq\{a',b'\}$ since $xy\notin E(G)$.
Without loss of generality, assume that $y=a'$ and let $f'$ lie in the pseudo face $F_2$.
Suppose that $V(\Theta_{f'})=\{a',b',a'',b''\}$.
If $V(\Theta_{f'})\neq\{a',b',u,v\}$ (say $a''\notin\{u,v\}$), then $f'\neq f$.
Since $f'$ lies in the pseudo face $F_2$, it follows that
$a'b'a''a'\cup\{au,av,uv,bu,bv\}$ is a   copy of $C_3\dot{\cup} \Theta_4$, a contradiction.
\end{proof}
Since $x,y$ belong to $\partial(F_1)$ and $\partial(F_2)$, respectively, it follows that for each $g\in B_e-f$, $\Theta_e\cup \Theta_g$ is not a copy of $D_2$. Further, $\Theta_e\cup \Theta_g$ cannot be isomorphic to $ D_3$; otherwise, $\Theta_e\cup \Theta_g\cup \Theta_f$ would contain a copy of $C_3\dot{\cup}  \Theta_4$, a contradiction.
Therefore, by combining this with Claim \ref{thm-3-clm-1}, we conclude that for each $g\in B_e$, $\Theta_e\cup \Theta_g$ is a copy of $D_1$.
\begin{claim}\label{thm-3-clm-2}
The following properties hold.
	\begin{enumerate}
	    \item $B_e$ is a matching.
        \item No edge $g\in B^*_e$ satisfies: 
        one endpoint of $g$ belongs to $\{u,v\}$ and the other endpoint does not belong to $V(\Theta_e)$.
	\end{enumerate}
\end{claim}
\begin{proof}
We prove the first statement.
Suppose, to the contrary, that $B_e$ is not a matching.
Without loss of generality, let $g=bc$ be an edge in $B_e$ distinct from $f$ where $f$ and $g$ share the vertex $b$.
Since $B_e$ contains at least 4 independent edges including $f$, there must be an edge $h\in B_e$ such that $h$ have no endpoints in $\{a,b,c\}$ (say $h=pq$).
Since $\Theta_f\cup \Theta_e$, $\Theta_g\cup \Theta_e$ and $\Theta_h\cup \Theta_e$ are copies of $D_1$, it follows that $\Theta_{xb}\cup  ypqy$ is a  copy of $C_3\dot{\cup}  \Theta_4$, a contradiction.

Now we show the second statement. Assume, for contradiction, that $g=uc$ is such an edge with $c\notin V(\Theta_e)$.
Since the two faces incident with $g$ are $3$-faces, it follows that $v\notin V(\Theta_g)$.
Since $B_e$ contains at least 4 independent edges, there is an edge $h\in B_e$ such that $V(\Theta_g)\cap V(h)=\emptyset$ (say $h=pq$).
Hence, $\Theta_g\cup vpqv$ is a $\Theta_4\dot{\cup}  C_3$, a contradiction.
\end{proof}

From Claim \ref{thm-3-clm-2}, $B_e$ is a matching with $B^*_e\subseteq G[V(\Theta_e)]$, thus $|B_e|\leq (n-2)/2$ and $|B^*_e|\leq e(G[V(\Theta_e)])$. Since $xy\notin E(G)$, we have $|B^*_e|\leq 5$. Consequently, $|E_I(G)|\leq |B_e|+|B^*_e|\leq \left\lfloor\frac{n}{2}\right\rfloor+4$, with equality only if $B_e$ is a matching of size $\left\lfloor\frac{n-2}{2}\right\rfloor$. Because $\Theta_e\cup \Theta_g$ forms a $D_1$ for each $g\in B_e$, $uv+B_e$ is a subgraph of $G$. Therefore, if $|E_I(G)|=\left\lfloor\frac{n}{2}\right\rfloor+4$, $G$ contains a subgraph isomorphic to $\left(\left\lfloor\frac{n-2}{2}\right\rfloor K_2\right )+K_2$.

\noindent\textbf{Case 2.} For every edge $e$ in $E_I(G)$, there is no edge $f$ in $A_e$ such that $\Theta_e\cup \Theta_f$ forms a copy of $D_1$.

We next prove that $|E_I(G)| < \frac{n}{2} + 4$.

\noindent{\bf Case 2.1.} An edge $f \in A_e$ exists such that $\Theta_e \cup \Theta_f$ is a copy of $D_2$.

Assume $f=ab$, $e=uv$, $V(\Theta_e)=\{u,v,x,y\}$, then $V(\Theta_f)=\{a,b,x,y\}$ as illustrated in Figure \ref{fig:Q} (b).
We claim $xy \notin E_I(G)$. Otherwise, if $xy \in E_I(G)$, then since $A_e$ contains at least four independent edges, we can choose $h$ such that $h$ and $\Theta_{xy}$ are vertex-disjoint. Thus, $\Theta_{xy} \cup \Theta_h$ forms a $D_1$, contradicting Case 2.
We now show that $A_e$ is a matching in $G$. It is enough to prove that any edge $g \in A_e$, where $g \neq f$, shares no vertex with $f$. We proceed by contradiction. Assume, without loss of generality, that $g=bq$ (where $b$ is the common vertex of $f$ and $g$).
Thus, $g$ is in the $byuxb$ pseudo face of $\Theta_e\cup \Theta_f$. As $|V(\Theta_e)\cap V(\Theta_g)|= 2$, $\Theta_{e}\cup \Theta_f\cup \Theta_g$ is isomorphic to $D_{1,1}$ or $D_{1,2}$ (Figure \ref{fig:D^*}). But $D_{1,1}$ and $D_{1,2}$ both contain $C_3\dot{\cup}  \Theta_4$, a contradiction. Hence, $A_e$ is a matching.

If there is an edge $e'\in A_e$ with $e_1=a_1b_1$ such that $\Theta_{e_1}\cup \Theta_e\cong D_3$. As $A_e$ is a matching, either $a_1b_1ya_1\cup \Theta_f$ or $a_1b_1xa_1\cup \Theta_f$ is a copy of $C_3\dot{\cup}  \Theta_4$, a contradiction. Therefore, we have that for each edge $e'\in A_e$, $\Theta_{e'} \cup \Theta_e\cong D_2$.

\begin{claim}
$|A_e^*|\leq 5$.
\end{claim}
\begin{proof}
We first prove that $|A_e^*-E(G[V(\Theta_e)])|\leq2$. 
Suppose $h\in A_e^*-E(G[V(\Theta_e)])$.
Without loss of generality, assume that $h=pq$ and $V(\Theta_h)=\{p,q,p',q'\}$, where $p\in V(\Theta_e)$ and $q\notin V(\Theta_e)$.
Since  $|A_e|\geq 4$, we can choose an edge $f'=a'b'$ from $A_e$ such that $V(\Theta_h)\cap\{a',b'\}=\emptyset$.
Note that $\Theta_{f'} \cup \Theta_e\cong D_2$.
If $p\in\{u,v\}$, then $\{p,q\}\cap \{a',b'\}=\emptyset$ ensures that $\Theta_{e}\cup \Theta_h\cup \Theta_{f'}$ contains a copy of  $\Theta_4\dot{\cup}  C_3$, a contradiction.
If $p\in\{x,y\}$ (say $p=x$), then $y\in \{p',q'\}$; otherwise $\Theta_h\cup ya'b'y$ is a $\Theta_4\dot{\cup} C_3$, a contradiction.
Therefore, for each such $h=pq$ of $A^*_e$, we have that $p\in\{x,y\}$. Moreover, $xqyx$ is a $3$-face of $G$ whenever $p=x$ or $p=y$.
This implies that $S=A_e^*-E(G[V(\Theta_e)])\subseteq\{qx,qy\}$.
Therefore,$|A_e^*-E(G[V(\Theta_e)])|\leq 2$.

Next, we complete the proof by considering two separated cases.
If $|S|=0$, then $A_e^*\subseteq E(G[V(\Theta_e)]-\{xy\})$, implying $|A_e^*|\leq 5$ (recall that $xy\notin E_I(G)$).
If $1\leq |S|\leq 2$, then $xy\in E(G)$ and $xqyx$ is a $3$-face of $G$.
We consider the plane graph $D=\Theta_e\cup \Theta_f\cup \{xy\}$ below.
Note that either $uxyu$ or $bxyb$ is a pseudo face of $D$ (without loss of generality, assume $F=uxyx$ is a pseudo face of $D$).
We claim that $ux,uy\notin E_I(G)$. Clearly, since $xy\notin E_I(G)$ and $xqyx$ is a $3$-face of $G$, it follows that $uxyu$ is not a $3$-face of $G$.
Therefore, if $ux\in E_I(G)$ or $uy\in E_I(G)$,  then $\Theta_{ux}\cup yaby$ or $\Theta_{uy}\cup xabx$ is a $\Theta_4\dot{\cup}  C_3$, a contradiction.
Therefore, $A_e^*\subseteq S\cup \{uv,vx,vy\}$, implying $|A_e^*|\leq 5$.
\end{proof}

By above discussion, we have that $A_e$ forms a matching and $|A_e^*|\leq 5$.
Therefore, $|E_I(G)|=|A_e^*|+|A_e|\leq \left\lfloor\frac{n-4}{2}\right\rfloor+5<\left\lfloor\frac{n}{2}\right\rfloor+4$.

\begin{figure}[ht] 
     \centering  
    \subfloat[{$D_{1,1}$}]
   {\begin{tikzpicture}
    \pgfmathparse{sqrt(3)}
    [inner sep=0.1mm]	
                 \node[circle, fill, inner sep=1.5pt](u) at (0,0)[]{};
               \node[circle, fill, inner sep=1.5pt](v) at (0.65,0)[]{};
                \node[circle, fill, inner sep=1.5pt](a) at (1.45,0)[]{};
                \node[circle, fill, inner sep=1.5pt](b) at (2.1,0)[]{};
                \node[circle, fill, inner sep=1.5pt](q) at (2.8,0)[]{};{};
                \node[circle, fill, inner sep=1.5pt](x) at (1.4,1.2)[]{};
                \node[circle, fill, inner sep=1.5pt](y) at (1.4,-1.2)[]{};
                \draw[-] (u) -- (v);
                \draw[-] (a) -- (b) -- (q);
                \foreach \point in {(u),(v),(a),(b), (q)} {
               \draw (x) -- \point;}
                \foreach \point in {(u),(v),(a),(b), (q)} {
               \draw (y) -- \point;}
               \node at (0.45,0.18) {$e$};
               \node at (1.7,0.18) {$f$};
               \node at (2.3,0.18) {$g$};
                \node at (-0.25,0) {$u$};
            \node at (0.9,0) {$v$};
            \node at (1.2,0) {$a$};
            \node at (2.15,-0.25) {$b$};
             \node at (3.05,0) {$q$};
            \node at (1.4,1.45) {$x$};
            \node at (1.4,-1.45) {$y$};
            \end{tikzpicture}
    } \hspace{3em}
     \subfloat[{$D_{1,2}$}]
    {\begin{tikzpicture}
    [inner sep=0.8pt]	
                 \node[circle, fill, inner sep=1.5pt](u) at (0,0)[]{};
               \node[circle, fill, inner sep=1.5pt](v) at (0.9,0)[]{};
                \node[circle, fill, inner sep=1.5pt](a) at (1.7,0)[]{};
                \node[circle, fill, inner sep=1.5pt](b) at (2.6,0)[]{};
                \node[circle, fill, inner sep=1.5pt](c) at (2.25,0.85)[]{};
                \node[circle, fill, inner sep=1.5pt](x) at (1.19,1)[]{};
                \node[circle, fill, inner sep=1.5pt](y) at (1.19,-1)[]{};
                \draw[-] (u) -- (v);
                \draw[-] (a) -- (b) -- (c) -- (x);
                \foreach \point in {(u),(v),(a),(b)} {
               \draw (x) -- \point;}
                \foreach \point in {(u),(v),(a),(b)} {
               \draw (y) -- \point;}
            \draw (0,0)arc (180:0:1.3 and 1.6);  
            \draw[-] (u) .. controls (0.9,1.5) .. (c);
            \node at (0.5,0.18) {$e$};
            \node at (1.95,0.18) {$f$};
            \node at (2.25,0.5) {$g$};
            \node at (-0.25,0) {$u$};
            \node at (1.15,0) {$v$};
            \node at (2.1,1.05) {$q$};
            \node at (1.45,0) {$a$};
             \node at (2.85,0) {$b$};
            \node at (0.98,1.1) {$x$};
            \node at (1.19,-1.25) {$y$};
            \end{tikzpicture}
             }  
              \caption{\label{fig:D^*} The plane graphs $D_{1,1}$ and $D_{1,2}$.}
    \end{figure}
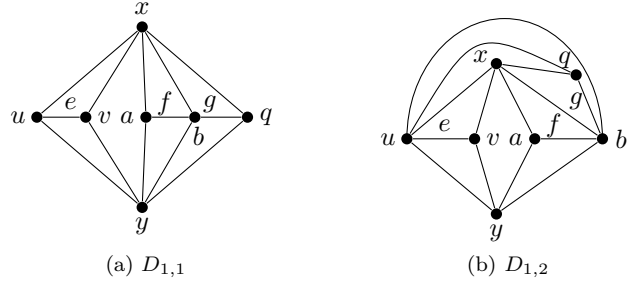

\noindent\textbf{Case 2.2.} For every edge $e$ in $E_I(G)$ and $f\in A_e$, $\Theta_e\cup \Theta_f$ is a copy of $D_3$. 

Since $A_e$ contains at least 4 independent edges, we choose two of them, say $f,g$.
It is clear that $\Theta_e\cup \Theta_f$ and $\Theta_e\cup \Theta_g$ are copies of $D_3$. Hence, $\Theta_f\cup \Theta_g$ is a copy of $D_1$, a contradiction.
\end{proof}

\noindent{\bf Proof of Theorem \ref{thm: C3Theta4}:}
Since $$2e(G)= \sum_{i}if_i(G)\geq 3f_3(G) + 4(f(G)-f_3(G))=3f_3(G)+4(e(G)+2-n-f_3(G)),$$
it follows that 
\begin{align}\label{thm-3-eq1}
2e(G)\leq f_3(G)+4n-8.
\end{align}
Let
$$E'=\{e: e \text{ lies on the boundary of exactly one 3-face of } G\}.$$
Then $E'\cap E_I(G) =\emptyset$ and thus $|E'|\leq e(G)-|E_I(G)|$. 
Since $f_3(G)=(|E'|+2|E_I(G)|)/3$, we have $f_3(G)\leq (e(G)+|E_I(G)|)/3$. 
Combining these inequalities, we obtain
\begin{align}\label{thm-3-eq2}
e(G)\leq \left\lfloor \frac{|E_I(G)|}{5}+\frac{12n}{5}-\frac{24}{5}\right\rfloor.
\end{align}
If $|E_I(G)|\leq \frac{n}{2}$, then $e(G)<\left\lfloor \frac{5n}{2}\right\rfloor-4.$
Hence, we assume that $|E_I(G)|> \frac{n}{2}$ below. 

\begin{claim} \label{thm3-clm1}
If $\Delta_I(G)\geq 10$ (say  $d_I(u) = \Delta_I(G)$), then $G-u$ is $C_3$-free. Moreover, $f_3(G)\leq n-1$.
\end{claim}
\begin{proof}
Let \( \Delta_I(G) = s \)  and  \( d_I(u) = s \). 
Then $s\geq 10$.
Suppose that \( d_G(u) = t \) and \( N_G(u) = \{u_0, u_1, \ldots, u_{t-1}\} \), where the vertices \( u_0, u_1, \ldots, u_{t-1} \) are listed in clockwise order around \( u \). We further let \( E_I(u) = \{uu_{c_0}, uu_{c_1}, \ldots, uu_{c_{s-1}}\} \). 
We first prove that every $C=C_3$ in \( G \) must contain \( u \). Suppose, for contradiction, that there exists a 3-cycle \( C \) containing no \( u \). Then, \( |V(C) \cap \{u_{c_i} \mid 0 \leq i \leq s-1\}| \leq 3 \). Since \( s \geq 10 \), there exists an index \( j \in \{0, 1, \ldots, s-1\} \) such that $
V(C) \cap \{u_{c_j - 1}, u_{c_j}, u_{c_j + 1}\} = \emptyset,
$  
where the subscripts are taken modulo \( t \). Consequently, the union \( C \cup \Theta_{uc_j} \) forms a \( C_3 \dot{\cup}  \Theta_4 \), which is a contradiction.
Hence, $G-u$ is $C_3$-free.
Furthermore, every $3$-face of $G$ contains $u$, implying $f_3(G)\leq n-1$.
\end{proof}

If $\Delta_I(G)\geq 10$, then by this Claim \ref{thm3-clm1} and Ineq. (\ref{thm-3-eq1}), we have that
\begin{align}\label{thm-3-eq3}
e(G) \leq \frac{f_3(G)}{2} + 2n - 4 \leq \frac{5n-1}{2}-4\leq\left\lfloor \frac{5n}{2} \right\rfloor - 4.
\end{align}
If $\Delta_I(G)\leq 9$, then since $n\geq 174$ and $|E_I(G)|> \frac{n}{2}$, by lemma \ref{lem:theupperboundof|E_I|} and Ineq. (\ref{thm-3-eq2}), we obtain
\begin{align}\label{thm-3-eq4}
e(G)\leq \left\lfloor \frac{|E_I(G)|}{5}+\frac{12n}{5}-\frac{24}{5}\right\rfloor\leq \left\lfloor \frac{5n}{2}\right\rfloor-4.
\end{align}
Therefore, $ex_{\mathcal{P}}(n, C_3\dot{\cup}  \Theta_4)\leq \lfloor \frac{5n}{2}\rfloor-4$.

To demonstrate tightness, we now characterize all the extremal graphs. 
Let $G_E$ be an $C_3\dot{\cup}  \Theta_4$-free plane graph with the maximum number of edges.
We will describe the characterization of $G_E$ under two different circumstances: $\Delta_I(G_E)\leq 9$ and $\Delta_I(G_E)\geq 10$.

If $\Delta_I(G_E)\leq 9$, then by Ineq. (\ref{thm-3-eq4}), 
$e(G_E)=\lfloor \frac{5n}{2}\rfloor-4$ if and only if $|E_I(G_E)|=\frac{n}{2}+4$.
By Lemma \ref{lem:theupperboundof|E_I|}, we conclude that $e(G_E)=\lfloor \frac{5n}{2}\rfloor-4$ implies $G_E$ contains a spanning subgraph $G'$ that is a copy of $\left(\left\lfloor\frac{n-2}{2}K_2\right\rfloor\right )+K_2$.
Without loss of generality, assume that $G'=xy+M$, where $M$ is a matching of size $\lfloor(n-2)/2\rfloor$.
If $n$ is even, then $e(G')=\lfloor \frac{5n}{2}\rfloor-4=e(G)$, and hence $G=G'$ is a copy of $\left(\frac{n-2}{2}K_2\right)+K_2$.
If $n$ is odd, then $|V(G)-V(G')|=1$ and $|E(G)-E(G')|=2$ (say $V(G)-V(G')=\{z\}$ and $E(G)-E(G')=\{e_1,e_2\}$).
It is clear neither $e_1$ nor $e_2$ belongs to $G[V(G')]$, for otherwise there is a $\Theta_4\dot{\cup}  C_3$, a contradiction.
Hence, we can assume that $e_1=z_1z$ and $e_2=z_2z$.
Then either $\{z_1,z_2\}=\{x,y\}$ or $z_1,z_2$ are endpoints of two edges in the matching $G'-\{x,y\}$, respectively, for otherwise there is a $\Theta_4\dot{\cup}  C_3$, a contradiction. 
Therefore, $G=G'\cup\{z_1z,z_2z\}$ is either a copy of $K_2+\left(\frac{n-2}{2}K_2\right)=K_2+M_{n-2}$ or a copy of $K_2\vee M_{n-2}$.
On the other hand,  $\left(\frac{n-2}{2}K_2\right)+K_2$ and $K_2\vee M_{n-2}$ are 
$\Theta_4\dot{\cup}  C_3$-free obviously.

If $\Delta_I(G_E)\geq 10$, then $e(G_E)<\lfloor \frac{5n}{2}\rfloor-4$ when $n$ is even. Hence, we consider the case where $\Delta_I(G_E)\geq 10$ and $n$ is odd. By the discussion above, $e(G_E)=\lfloor \frac{5n}{2}\rfloor-4$ if and only if $G_E$ satisfies the following two conditions.
\begin{enumerate}
    \item There exists a vertex $u\in V(G_E)$ that belongs to $n-1$ 3-faces of $G_E$, and $f_3(G_E)=n-1$;
    \item $n$ is odd and $G-u$ is a $C_3$-free outerplanar graph with $ex_{\mathcal{OP}}(n-1,C_3)=\lfloor \frac{3n}{2}\rfloor-3$ edges.
\end{enumerate}
Therefore, $G_E=u+O$, where  $n=|G_E|$ is even and $O=G_E-u$ is a $C_3$-free outerplanar graph with $ex_{\mathcal{OP}}(n,C_3)=\lfloor \frac{3n}{2}\rfloor-5$ edges.
On the other hand, for any planar graph $G=u+O$ with $O=$ a $C_3$-free outerplanar graph, we can easily to verify that $G$ is $\Theta_4\dot{\cup}  C_3$-free.

\section{Acknowledgments}

Ping Li is suppose by National Natural Science Foundation of China (No. 12201375); 
Xuqing Bai is supported by the Fundamental Research Funds for the Central Universities (No.\,ZYTS24069); Zhipeng Gao is supported by the Natural Science Basic Research Program of Shaanxi (No.\,2023-JC-YB-001)
and the Fundamental Research Funds for the Central Universities (No.\,ZYTS24076).

\bibliographystyle{plain}
\bibliography{ref}
\end{document}